\documentclass{alggeom}

\usepackage[T1]{fontenc}
\usepackage{times}
\usepackage{verbatim}
\usepackage[utf8]{inputenc}
\usepackage[english]{babel}

\usepackage{amscd}
\usepackage{amsmath}
\usepackage{amssymb}
\usepackage{amstext}
\usepackage{amsthm}
\usepackage{graphics}
\usepackage{graphicx}
\usepackage{hyperref}
\usepackage{ifmtarg}
\usepackage{mathrsfs}
\usepackage{mathtools}
\usepackage{multirow}
\usepackage{tikz}

\theoremstyle{plain}

\newtheorem{thm}{Theorem}[section]
\newtheorem*{conj*}{Conjecture}
\newtheorem{conj}{Conjecture}[section]
\newtheorem{dfn}[thm]{Definition}
\newtheorem{lemma}[thm]{Lemma}
\newtheorem{prop}[thm]{Proposition}
\newtheorem{cor}[thm]{Corollary}
\newtheorem{problem}[thm]{Problem}

\newtheorem{THM}{Theorem}

\newtheorem*{acknowledgments}{Acknowledgments}

\newtheorem{example}[thm]{Example}
\newtheorem{remark}[thm]{Remark}
\newtheorem{notation}[thm]{Notation}

\addtocounter{section}{0}             
\numberwithin{equation}{section}       

\newcommand{\mb}{\mathbb}
\newcommand{\mc}{\mathcal}
\newcommand{\R}{\mb R}
\newcommand{\C}{\mb C}

\newcommand{\Q}{\mb Q}

\newcommand{\F}{\mc F}
\newcommand{\G}{\mc G}
\newcommand{\FF}{\mathscr F}
\newcommand{\GG}{\mathscr G}
\newcommand{\XX}{\mathscr X}
\newcommand{\YY}{\mathscr Y}

\newcommand{\px}{\frac{\partial}{\partial x}}
\newcommand{\py}{\frac{\partial}{\partial y}}

\newcommand{\KF}{{K_\F}}
\newcommand{\TF}{{T_\F}}

\newcommand{\KG}{{K_\G}}
\newcommand{\KX}{{K_X}}
\newcommand{\KY}{{K_Y}}
\newcommand{\KZ}{{K_Z}}
\newcommand{\CNF}{{N^*_\F}}
\newcommand{\CNG}{{N^*_\G}}

\DeclareMathOperator{\sing}{sing}

\DeclareMathOperator{\ord}{ord}

\DeclareMathOperator{\ii}{i}
\DeclareMathOperator{\eff}{eff}
\DeclareMathOperator{\Aut}{Aut}

\DeclareMathOperator{\adj}{adj}
\DeclareMathOperator{\adjnum}{adj_{num}}
\DeclareMathOperator{\kod}{kod}

\DeclareMathOperator{\Res}{Res}

\makeatletter

\newcommand\contFrac{\@ifstar{\@contFracStar}{\@contFracNoStar}}

\def\singleContFrac#1#2{%
	\begin{array}{@{}c@{}}%
		\multicolumn{1}{c|}{#1}%
		\\%
		\hline%
		\multicolumn{1}{|c}{#2}%
	\end{array}%
}

\def\@contFracNoStar#1{%
	\mathchoice{
		\@contFracNoStarDisplay@#1//\@nil%
	}{
	\@contFracNoStarInline@#1//\@nil%
}{
\@contFracNoStarInline@#1//\@nil%
}{
\@contFracNoStarInline@#1//\@nil%
}%
}

\def\@contFracNoStarDisplay@#1//#2\@nil{%
	\@ifmtarg{#2}{%
		#1%
	}{%
	#1+\cfrac{1}{\@contFracNoStarDisplay@#2\@nil}%
}%
}

\def\@contFracNoStarInline@#1//#2\@nil{%
	\@ifmtarg{#2}{%
		#1%
	}{%
	#1 \@@contFracNoStarInline@@#2\@nil%
}%
}
\def\@@contFracNoStarInline@@#1//#2\@nil{%
	\@ifmtarg{#2}{%
		+ \singleContFrac{1}{#1}%
	}{%
	+ \singleContFrac{1}{#1} \@@contFracNoStarInline@@#2\@nil%
}%
}

\def\@contFracStar#1{%
	\mathchoice{
		\@contFracStarDisplay@#1////\@nil%
	}{
	\@contFracStarInline@#1//\@nil%
}{
\@contFracStarInline@#1//\@nil%
}{
\@contFracStarInline@#1//\@nil%
}%
}

\def\@contFracStarDisplay@#1//#2//#3\@nil{%
	\@ifmtarg{#2}{%
		#1%
	}{%
	#1 + \cfrac{#2}{\@contFracStarDisplay@#3\@nil}%
}%
}

\def\@contFracStarInline@#1//#2\@nil{%
	\@ifmtarg{#2}{%
		#1%
	}{%
	#1 \@@contFracStarInline@@#2\@nil%
}%
}
\def\@@contFracStarInline@@#1//#2//#3\@nil{%
	\@ifmtarg{#3}{%
		+ \singleContFrac{#1}{#2}%
	}{%
	+ \singleContFrac{#1}{#2} \@@contFracStarInline@@#3\@nil%
}%
}
\makeatother

\title{Effective algebraic integration in bounded genus}

\author{Jorge Vit\'{o}rio Pereira}
\email{jvp@impa.br}
\address{IMPA, Estrada Dona Castorina, 110, Horto, Rio de Janeiro,
	Brasil}

\author{Roberto  Svaldi}
\email{RSvaldi@dpmms.cam.ac.uk}
\address{DPMMS\\
Centre for Mathematical Sciences\\
University of Cambridge\\
Wilberforce Road\\ Cambridge\\ CB3 0WB\\
United Kingdom}

\subjclass[2010]{37F75; 14E99}
\keywords{Holomorphic foliations; effective algebraic integration; degree of invariant algebraic curves.}
\thanks{This collaboration initiated while both authors where visiting
James M\textsuperscript{c}Kernan at UCSD, and continued during a visit of the second author to IMPA.
We are grateful to both institutions for the  favorable working  conditions.
The first author is partially supported by Cnpq and FAPERJ.
The second author was partially supported by NSF research grant no: 1200656 and no: 1265263.
During the final revision of this work  he was supported by funding from the
European Union's Seventh Framework Programme (FP7/2007-2013)/ERC
Grant agreement no. 307119.}
\begin{document}

\begin{abstract}
We introduce and study  birational invariants for foliations on
projective surfaces built from the adjoint linear series of
positive powers of the canonical bundle of the foliation.
We apply the  results in order to investigate the effective
algebraic integration of foliations on the projective plane.
In particular, we describe the Zariski closure of the set $\Sigma_{d,g}$ of foliations on
$\mathbb P^2$ of degree $d$ admitting rational first integrals with
fibers having geometric genus bounded by  $g$.
\end{abstract}

\maketitle

\setcounter{tocdepth}{1}

\tableofcontents

\section{Introduction}

\subsection{Effective algebraic integration}

It seems fair to say that the simplest class of algebraic ordinary differential equations
consists of the class of equations having all its  solutions algebraic. In general, given an explicit
differential equation, it is a difficult problem to decide whether or not it belongs to this
distinguished class. Perhaps the first positive result on the subject is Schwarz's
list of parameters for which Gauss' hypergeometric equation has an algebraic solution
\cite{Schwarz1873}.

Motivated by this remarkable result, a lot of activity on the study
of algebraic solutions of linear differential equations took place in the XIXth century
leading to a fairly good understanding of the problem for homogeneous
linear differential equations. Among the works dealing with this question one can
find contributions by Fuchs, Gordan, Jordan, Halphen, and Klein just to name a few.
At that time, the community seemed to believe that it would be possible to decide
whether or not all solutions of a given linear differential equations are algebraic, see for
instance the concluding remarks\begin{footnote}
	{ ``Thus is the problem, which we formulated at the beginning of this paragraph
		[{present all linear homogenous differential equations of
			the second order with rational coefficients: $y'' + p y' + q y = 0$ which possess
			altogether algebraic solutions}], fully solved.''}
\end{footnote} of  \cite[Section 3, Chapter V]{MR0080930}.

By the end of XIXth century mathematicians like Painlev\'e, Autonne, and Poincar\'e
\cite{PoincarePalermoI, zbMATH02673353} started to study the next case, that is,
polynomial differential equations of first order and of first degree. In modern language,
they studied foliations on the projective plane with special emphasis on the existence
of methods/algorithms to decide whether or not all leaves are algebraic.
We will call this general line of enquiry effective algebraic integration. The results
obtained at that time relied on strong assumptions on the nature of the singularities
of the foliations and were not considered definitive as one
can learn from the Introduction\begin{footnote}{
		``Je me suis occup\'{e} de nouveau de la m\^{e}me question dans ces derniers
		temps, dan l' espoir que je parviendrais \`{a} g\'{e}n\'{e}raliser les r\'{e}sultats
		obtenus. Cet espoir a \'{e}t\'{e} d\'{e}\c{c}u.
		J'ai obtenu cependant quelques r\'{e}sultats partiels, que je prends la libert\'{e} de
		publier, estimant qu'on pourra s'en servir plus tard pour obtenir, par un nouvel effort,
		une solution plus satisfaisante du probl\`{e}me.''}
\end{footnote}
of \cite{zbMATH02673353}. For a modern account of some of these classical results
see \cite{MR1485488} and \cite[Chapter 7]{MR2029287}.

The results of the XIXth century on the effective integration of linear differential
equations were revisited in the course of the XXth century. It was then made
clear that a full solution for the problem was not available, rather, the problem was reduced
to a similar one for rank one linear differential equations over curves.
More precisely, in order to be able to decide whether or not a homogeneous linear
differential equation $P(x,y,y',y'',y''', \ldots, y^{(n)})=0$ has all its solutions algebraic
it suffices to be able to solve the following problem: given an element $u$
belonging to an algebraic extension of the field $\mathbb C(x)$, decide if $u$
is the logarithmic derivative of an element $v$ also belonging to an algebraic
extension of $\mathbb C(x)$.
Some authors expressed doubts on the possibility of solving this problem.
For instance, in \cite[page 51]{MR0349924} one can find the  view of
Hardy\begin{footnote}
	{``But no method has been devised as yet by which we can always determine in a
		finite number of steps whether a given elliptic integral is pseudo-elliptic, and
		integrate it if it is, and there is reason to suppose that no such method can be given.''}
\end{footnote} on the subject.

Despite the skepticism of Hardy and others (cf.  \cite{MR0269635}), in the late 1960's
Risch (loc. cit.) showed  that this problem, in its turn, can be reduced to the following one:
given an explicit divisor  on an explicit algebraic curve $C$, decide whether or not such
divisor is of finite order in the Jacobian of $C$. Risch proved that this problem can be
solved by restricting the data modulo two distinct primes and using the resulting bounds
in positive characteristic to devise an explicit bound in characteristic zero.
For a detailed account on the case of second order homogeneous differential equations
see \cite{MR527825}. The interested reader can find more about the history of effective algebraic
integration of linear differential equations in \cite[page 124]{MR1960772},
\cite[Chapter III]{MR2414794}, and the references therein.

The corresponding problem for (non-linear) differential equations of the first order and
of the first degree is still wide open and received considerably less attention.
After being dormant for a good while, the interest towards it has been revived by experts
in foliation theory who considered the problem of bounding the degree of algebraic leaves
of foliations on $\mathbb P^2$, see for instance
\cite{MR1150571,MR1298714,MR1407696,MR1993042} and references therein.
The influence of arithmetic on the subject was rediscovered by Lins Neto \cite{MR1914932}
who determined algebraic families (pencils) of foliations on the projective plane with
fixed number and analytical type of singularities and with algebraic leaves of
arbitrarily large degree.

\subsection{Degenerations of planar foliations admitting a rational first integrals}

This work investigates the problem of effective algebraic integration
for foliations on projective surfaces. In order to focus the discussion and
clarify the framework in which we are going to carry it, we introduce the following
conjecture.

\begin{conj}\label{Conj}
	The Zariski closure in $\mathbb P H^0(\mathbb P^2, T_{\mathbb P^2}(d-1))$
	of the set of foliations of degree $d$ on $\mathbb P^2$ which admit a rational integral
	consists of transversely projective foliations.
\end{conj}

This conjecture is inspired by a remark made by Painlev\'{e}\begin{footnote}
	{``J’ajoute qu’on ne peut esp\'erer r\'esoudre d’un coup qui consiste \`a limiter
		$n$. L'\'enonc\'e vers lequel il faut tendre doit avoir la forme suivante:
		{\it ``On sait reconna\^itre si l’int\'egrale d’une  \'equation $F(y\prime,y,x)=0$
			donn\'ee est alg\'ebrique ou ramener l’\'equation aux quadratures.''}
		Dans ce dernier cas, la question reviendrait \`a reconna\^itre si une certaine
		int\'egrale ab\'elienne (de premi\`ere ou de troisi\`eme esp\`ece) n’a que
		deux ou une p\'eriodes.''}  \end{footnote} (\cite[pp. 216--217]{zbMATH02673308})
in his Stockholm's lectures. Knowledge of a transversely projective structure for a
given foliation, in view of their recent description \cite{MR3294560, MR3522824},
would allow to reduce the problem to either the determination of periods of
differential forms -- when, after passing to a ramified covering, the foliation is
defined by a closed rational $1$-form -- or to the algebraic integrability
of Riccati  equations.

The main results of this paper provide evidence in favor of this conjecture and
are obtained using birational techniques. More precisely, we use basic results
on adjoint linear series, the birational classification of foliated
surfaces according to their Kodaira dimension
\cite{MR2435846, MR2114696, MR1785264} to obtain a variant of the classification which
we now proceed to explain.

\subsection{Adjoint dimension of foliations} The works of the Italian school
of algebraic geometry in the beginning of the XXth century
showed how much of the geometry of a smooth projective surface $X$ can be
determined by the order of growth of the function
\[
n \mapsto h^0(X, \KX^{\otimes n}).
\]
Whenever this function grows slower than a quadratic polynomial,
one has a rather precise description of the surface (the so called Enriques-Kodaira
classification). A similar classification is also available in dimension three thanks
to the works of the modern school of birational geometry, and
there is also a similar picture in arbitrary dimensions conditional on the
so-called Abundance Conjecture.

In the case of foliations on surfaces, McQuillan, Brunella and Mendes
obtained a very precise classification -- analogue to the Enriques-Kodaira
classification -- in terms of the Kodaira dimension of the foliation.
As in the case of surfaces, the Kodaira dimension of a foliation $\F$, $\kod(\F)$,
measures the growth of the function $h^0(X, \KF^{\otimes n})$ where $\KF$
is the bundle of holomorphic $1$-forms along the leaves of the foliation.

As the terminology suggests the canonical bundle
together with its dual are the most obvious  naturally determined line bundles
on a variety. Combined with the fact that the integers $h^0(X,\KX^{\otimes n})$ ($n>0$)
are birational invariants for smooth projective varieties, its study is rather
natural if one wants to understand varieties birationally.
For  foliations of arbitrary dimension/codimension, besides the canonical bundle,
one also has another naturally attached line bundle: the determinant
of the conormal bundle.
If $\F$ is a foliation on a projective surface $X$ with canonical singularities
then it turns out that for arbitrary $n,m\ge0$ the integers
$h^0(X, \KF^{\otimes n}\otimes\CNF^{\otimes m})$ are birational invariants.
Most of the results obtained in this paper steam from this simple observation.
We define the adjoint dimension of a foliation according to the
order of growth of the function $h^0(X, \KF^{\otimes n}\otimes\CNF^{\otimes m})$,
see Section \ref{S:adjoint}.

Building on the classification of foliations on surfaces according to their
Kodaira dimension, in Section \ref{S:classification} we present a
classification in function of the adjoint dimension. The results we obtain are
summarized in Table \ref{Table:1}.
The outcome of the classification provides a framework well-suited
to deal with families of foliations (Section \ref{S:variation}) mainly due to
the fact that it is more flexible with respect
to the type of singularities which are allowed
(Section \ref{S:singularities}).
The classification in terms of the adjoint dimension also reflects
distinct cases of the problem of effective algebraic integration (Section \ref{S:rationalfirstintegral}).

\begin{table}[h!]
	\centering
	\begin{tabular}{|l|l|l|}
		\hline
		$\adj $ & $\kod $  & Description  \\
		\hline \hline
		$-\infty$ &   $-\infty$ & Rational fibration \\
		\cline{2-3}
		& $0$   & Finite quotient of  Riccati foliation  generated by global vector field\\
		\cline{2-3}
		& $1$   &  Riccati foliation  \\
		\hline \hline
		$0$ &   $0$ & Finite quotient of   linear foliation on a torus \\
		\hline \hline
		$1$  &   $0$ &  Finite quotient of $E\times C \to C$, $g(C) \ge 2$   \\
		\cline{2-3}
		&   $1$ &  Finite quotient of $E\times C \to E$, $g(C) \ge 2$\\
		\cline{2-3}
		&   $1$ & Turbulent foliation \\
		\cline{2-3}
		&  $1$ &  Non-isotrivial elliptic fibration \\
		\hline \hline
		$2$ &  $-\infty$ & Irreducible quotient of $\mathbb H \times \mathbb H \to \mathbb H$ \\
		\cline{2-3}
		& $1$   &  Finite quotient of $C_1 \times C_2\to C_1$, $g(C_i)\ge 2$ \\
		\cline{2-3}
		& $2$   & General type\\
		\hline
	\end{tabular}
	\medskip
	\caption{Classification of foliations according to their adjoint/Kodaira dimensions.}
	\label{Table:1}
\end{table}

\subsection{Plan of the paper and  statement of main results}

The bulk of the paper starts by reviewing the classification of foliations with respect to
their Kodaira dimension in Section \ref{S:Kodaira}.
Then we introduce new birational invariants for foliations on surfaces,
notably the effective threshold and the adjoint dimension, in Section \ref{S:adjoint}.
Section \ref{S:singularities} is devoted to the study of a variation of the concept of
canonical singularities, the so-called $\varepsilon$-canonical singularities. We prove in
Corollary \ref{C:epsilon open} that, for $\varepsilon>0$, this concept is stable for
small perturbations of the singularity of the foliation.
This fact will be particularly important in the study of families of
foliations carried out in Section \ref{S:variation}.

Section \ref{S:Painleve} is devoted to the proof of the boundedness of non-isotrivial
fibrations of bounded genus in families, see Theorem \ref{T:A}.
In the particular case of $\mathbb P^2$, the result reads as follows.

\begin{THM}\label{THM:A}
	Let $\F$ be a foliation on $\mathbb P^2$. Assume that $\F$ is birationally
	equivalent to a non-isotrivial fibration of genus $g\ge 2$.
	Then  the degree of the general leaf
	of $\F$ is bounded by
	\[
	\Big(4 \Big(42(2g-2)\Big)! \Big)^2 (4g -4) \deg(\F).
	\]
\end{THM}

Theorem \ref{THM:A} refines the main result of \cite{MR1913040}
where it was established the existence of a bound for the degree
of the general leaf depending on its genus and on the first $k>0$
for which the linear system $|\KF^{\otimes k}|$ defines a rational
map with two dimensional image.
The existence of universal $k$ working for every non-isotrivial fibration
of genus $g$ was not known then - and is still not known at present time -
hence the existence of a bound depending only on the degree of the
foliation and on the genus was unclear.
In comparison to \cite{MR1913040} the proof of the result above has two new ingredients.
The first is a bound on multiplicities of irreducible components of fibers
of relatively minimal non-isotrivial fibrations of genus $g\ge 2$, cf. Proposition \ref{index.prop}.
The second new ingredient is the use of standard results on adjoint linear series
(recalled in Section \ref{S:producing}) in order to obtain effective $(n,m) \in \mathbb{N}^2$
such that the rational map defined by $|\KF^{\otimes n} \otimes \KX^{\otimes m}|$
has two dimensional image. By imposing further assumptions on the nature
of the singularities of a foliation on $\mathbb P^2$ we obtain significantly better
bounds (sub-linear on $g$),
refining a classical result of Poincar\'{e}, cf. Theorem \ref{T:Poincare}.

In Section \ref{S:classification} we carry out the classification of foliations
on surfaces according to the adjoint dimension, see Table \ref{Table:1}.
The proof strongly relies on the classification of foliations according to the
Kodaira dimension, but it does need to dwell with its subtlest point: the
classification of non--abundant foliations.
A nice corollary of the classification is a cohomological characterization of
rational fibrations, which is a weak analogue of Castelnuovo's Criterion for
the rationality of surfaces, cf. \cite[Thm. V.1]{MR1406314}.

\begin{THM}\label{THM:C}
	Let $\F$ be a foliation with at worst canonical singularities
	on a smooth projective surface $X$.
	The foliation $\F$ is a rational fibration if, and only if,
	$h^0(X,\KF^{\otimes n}\otimes \CNF^{\otimes m})=0$ for every
	$n\ge 1$ and every $m>0$.
\end{THM}

Section \ref{S:variation} investigates families of foliations.
There it is shown that the set of effective thresholds in a family does not
accumulate at zero (Theorem \ref{T:non accumulation}).
More important, it prepares the ground for the proof of the most
compelling evidence we have so far in favor of Conjecture \ref{Conj}.

\begin{THM}\label{THM:B}
	The Zariski closure in $\mathbb P(H^0(\mathbb P^2, T_{\mathbb P^2}(d-1)))$
	of the set of degree $d$ foliations admitting a rational first
	integral with general fiber of genus $\le g$ is formed by
	transversely projective foliations.
\end{THM}

Its proof  is presented in Section \ref{S:rationalfirstintegral} and
relies on Theorem \ref{THM:A}, on the birational classification of
foliations, and on basic properties of families of foliations.

\begin{acknowledgments} 
The authors wish to thank Calum Spicer for many valuable conversations, 
Carlo Gasbarri for suggesting an improvement
of the bounds in Theorem \ref{THM:A}, and the anonymous referees for
suggesting many improvements in the presentation of this work.
\end{acknowledgments}

\section{Kodaira dimension of foliations}\label{S:Kodaira}

We start things off by reviewing the birational classification of foliations on surfaces
following \cite{MR2435846} and \cite{MR2114696}.  No new results are presented in this section.
We have only included proofs of a few key properties of the Zariski decomposition of
the canonical bundle of a foliation which will be used in the sequel.

\subsection{Singularities of foliations}

\begin{dfn}
	Let $\F$ be a foliation on $X$ and let $\pi : Y \to X$ be a birational morphism.
	Denote by $\G$ the pull-back of $\F$ under $\pi$.
	If $E$ is an exceptional divisor of $\pi$ then the discrepancy of $\F$ along $E$ is
	\[
	a(\F,E) = \ord_E ( \KG - \pi^* \KF )  \, .
	\]
\end{dfn}

\begin{dfn}
	Let $\F$ be a foliation on $X$.
	A point $x \in X$ is canonical for $\F$ if and only if $a(\F,E) \ge 0$ for every divisor $E$ over $x$.
	A point $x \in X$ is log canonical for $\F$ if and only $a(\F,E) \ge -1$ for every divisor $E$ over $x$.
\end{dfn}

\begin{example}
	Consider the pencil of foliations on $X=\mathbb P^2$ defined by the vector fields
	$s x\frac{\partial}{\partial x} + t y \frac{\partial}{\partial y}$ where
	$(s:t) \in \mathbb P^1$.  If $s\cdot t \cdot (s-t) \neq 0$ then $\F_{(s:t)}$ is a
	foliation with trivial canonical bundle and three singularities
	at the points $(0:0:1), (0:1:0),$ and $(1:0:0)$.
	For $(s:t) \notin \mathbb P^1(\mathbb Q)$ the three  singularities  are canonical.
	For $(s:t) \in \mathbb P^1(\mathbb Q) - \{(0:1),(1:0),(1:1)\}$,
	two of the singularities  are log canonical but not canonical,
	while the third singularity is canonical. 	Finally, when
	$s\cdot t \cdot (s-t)=0$, the vector field will have one of the coordinate axis as a line
	of singularities. The corresponding foliation will have canonical bundle
	$\mathcal O_{\mathbb P^2}(-1)$ and only one singularity which is log canonical
	but not canonical.
\end{example}

Any foliation on a projective surface is birationally equivalent to a foliation having at worst canonical
singularities thanks to the following which is essentially due to Seidenberg.

\begin{thm}
	Let $\F$ be a foliation on a smooth projective surface $X$. Then there
	exists a finite composition of blow-ups $\pi : Y \to X$ such that all the singularities
	of $\pi^* \F$ are canonical.
\end{thm}

\subsection{Kodaira dimension}

\begin{dfn}
	Let $\F$ be a foliation with at worst canonical singularities on a smooth projective surface $X$.
	The Kodaira dimension of $\F$, $\kod(\F)$, is by definition
	\[
	\kod(\F) := \kod(\KF) = \max_{m \in \mathbb{N}} \{ \dim \phi_m(X)\},
	\]
	where $\phi_m \colon X \dashrightarrow \mathbb{P}(H^0(X, \KF^{\otimes m})^\ast)$ and
	we adopt the convention that $\dim \phi_m(X) = -\infty$ when $h^0(X, \KF^{\otimes m})=0$.
	(and it is not possible to define the associated map).
	
	The numerical Kodaira dimension of $\F$, $\nu(\F)$,
	is defined to be the numerical dimension of $\KF$, that is:
	\begin{itemize}
		\item $\nu(\F) = -\infty$ if $\KF$ is not pseudoeffective, while
		\item if $\KF$ is pseudoeffective with Zariski decomposition $\KF=P+N$
		then $\nu(\F)=0$ if $P$ is numerically zero, $\nu(\F) =1$ if $P\neq 0$ but $P^2=0$,
		and $\nu(\F) =2$ if $P^2>0$.
	\end{itemize}
\end{dfn}

The classification of foliations with negative numerical Kodaira dimension is
stated in the next result is due to Miyaoka.

\begin{thm}\label{T:non.pseff.fibr}
	Let $\F$ be a foliation on a projective surface $X$.
	If $\KF$ is not pseudoeffective then $\F$ is birationally equivalent
	to a $\mathbb P^1$-bundle over a curve.
\end{thm}

\subsection{Relatively minimal models}

\begin{dfn}
	Let $\F$ be a foliation with canonical singularities on a smooth projective surface $X$.
	An irreducible curve $C\subset X$ is called $\F$-exceptional if
	$\KX \cdot C = -1$ (i.e. $C \simeq \mathbb P^1$ and $C^2= -1$)
	and the contraction of $C$ gives rise to a foliation with canonical singularities.
\end{dfn}

\begin{dfn}
	Let $\F$ be a foliation with canonical singularities on a smooth projective surface $X$.
	A relatively minimal model for $\F$ is the datum of a foliation $\G$  with canonical singularities
	and without $\G$-exceptional curves  on a smooth projective surface $Y$
	which is birationally equivalent to $\F$.
	We say that $\G$ is a minimal model if for any birational map $\pi: Z \dashrightarrow Y$
	and any foliation $\mathcal H$ on $Z$ with canonical singularities such that
	$\pi_\ast \mathcal{H}=\G, \; \pi$ is a birational morphism.
\end{dfn}

The definitions above and the next result are essentially due to Brunella \cite{MR1708643}.
The only minor difference is that in the original definition of $\F$-exceptional curve
Brunella only considered reduced singularities instead of canonical singularities.
Nonetheless, his proof works also in this slightly more general situation.

\begin{thm}
	Let $\F$ be a foliation with at worst canonical singularities on a smooth surface $X$.
	There exists a birational morphism $\pi: X \to Y$ such that $\pi_\ast \F$
	is a relatively minimal model for $\F$.
	Moreover, $\pi_\ast \F$ is a minimal model for $\F$
	unless $\F$ is birationally equivalent to a rational fibration,
	a  Riccati foliation, or Brunella's special foliation $\mathcal H$.
\end{thm}

The reader will find the explicit construction of the foliation $\mathcal H$ from the theorem
in the paper just cited.

\begin{remark}
	The above theorem highlights one of the differences between the birational classification of
	projective surfaces and that of foliations on surfaces: while surfaces of non-negative Kodaira
	dimension always have a unique minimal model, there are foliations of Kodaira dimension
	zero and one which do not have unique minimal models.
\end{remark}

\subsection{Zariski decomposition and nef models}

If $\mathcal L$ is a pseudoeffective line bundle on a smooth projective surface
then $\mathcal L$ is $\Q$-linearly equivalent to $P_{\mathcal L} + N_{\mathcal L}$
where $P_{\mathcal L}$ is a nef $\mathbb Q$-divisor and $N_{\mathcal L}$ is a contractible
effective $\mathbb Q$-divisor satisfying $P_{\mathcal L} \cdot N_{\mathcal L}=0$.
This is the so-called Zariski decomposition of $\mathcal L$. We will denote by $\ii(\F)$
the index of $\KF$, i.e., the minimum of the set $\{n \in \mathbb N \;| \; nN \text{ has integral
	coefficients} \}$.

\begin{thm}\label{T:McQuillan nef}\cite[Chapter 8, Theorem 1]{MR2114696}
	Let  $\mathcal F$ be a relatively minimal foliation on a smooth projective surface $X$.
	If $\KF$ is pseudoeffective and $P + N$
	is its Zariski decomposition  then the support of $N$ is a disjoint union of Hirzebruch-Jung strings.
\end{thm}

A Hirzebruch-Jung string  is a chain of smooth rational curves of self-intersection at most $-2$.
At one end of the chain, the handle of the Hirzebruch-Jung string, the foliation has only one singularity.
Every other curve in the chain contains two singularities of the foliation.
There is only one singularity of $\F$ on the Hirzebruch-Jung string which does not
coincide with a singularity of its support. There exists a unique leaf of $\F$ not
contained in the Hirzebruch-Jung string that passes through this singularity.
Such curve is called the tail of the Hirzebruch-Jung string.

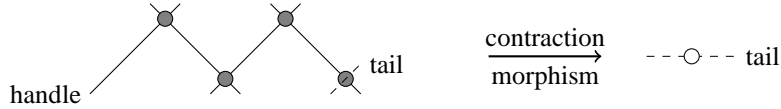
\begin{center}
	\begin{tikzpicture}
	\draw (0,0) node[left] {handle}--(1.2,1.2) ;
	\draw (0.8,1.2)--(2,0);
	\draw (1.6,0)--(2.8,1.2);
	\draw (2.4,1.2)--(3.6,0);
	
	\draw[fill=gray] (1,1) circle(.1);
	\draw[fill=gray] (1.8,0.2) circle(.1);
	\draw[fill=gray] (2.6,1) circle(.1);
	\draw[fill=gray] (3.4,0.2) circle(.1);
	
	\draw[dashed,very thin] (3.2,0) -- (3.6,0.4) node[right] {tail};
	
	\draw[thick,->] (5.3,0.5)  to  node[above] {contraction} node[below] {morphism} (6.8,0.5) ;

	\draw[dashed, very thin] (7.4,0.5)--(8.6,0.5) node[right] {tail};
	\draw[fill=white] (8,0.5) circle(.1) ;
	
	\end{tikzpicture}
\end{center}

\begin{dfn}
	Let $\F$ be a relatively minimal foliation with pseudoeffective $\KF$
	on a smooth projective surface $X$.
	The order of a maximal Hirzebruch-Jung string contained in the support of $N$
	is the determinant of the negative of the intersection matrix of its support.
\end{dfn}

The following proposition shows that the order and the index are closely related.

\begin{prop}\label{P:order HJ}
	Notation as in the definition above.  The following assertions hold true.
	\begin{enumerate}
		\item  The order of a maximal Hirzebruch-Jung string $J$ contained in the support of
		$N$ coincides with the smallest $o \in \mathbb N$ such that the coefficients of $N$
		corresponding to curves in $J$ belong to $\frac{1}{o} \mathbb N$.
		\item The contraction of a Hirzebruch-Jung string of order $o$ is locally isomorphic
		to the quotient of a smooth foliation on $(\mathbb C^2,0)$ by the cyclic group
		generated by an automorphism of the form
		$(x,y)\mapsto (\xi_o \cdot x, \xi_o^a \cdot y )$ where $\xi_o$ is a primitive root of
		unity of order $o$ and $a$ is a natural number relatively prime to $o$.
	\end{enumerate}
\end{prop}

\begin{proof}
	The statement is local so we may very well assume that the support of $N$ is connected.
	Let us write $N = \sum_{i=1}^k a_i E_i$ where $E_i$ are the irreducible components of $N$.
	We denote by $E_1$ the handle of the Hirzebruch-Jung string while the other
	curves are numbered following the order in which they appear in the chain.
	
	Let $A=(E_i \cdot E_j)_{i,j}$ be the intersection matrix of the Hirzebruch-Jung string
	and let $o=\det(-A)$ be the order of the Hirzebruch-Jung string. To determine the coefficients
	$a_1, \ldots, a_k$ we have to solve the linear system $(-A)\cdot (a_1, a_2, \ldots, a_k)^T =
	(1, 0, \ldots, 0)^T$. Therefore the coefficients $a_i$ certainly lie in 	$\frac{1}{o} \mathbb N$.
	To see that $o$ is the minimal number with such property it suffices to notice that
	$a_k = 1/o$, cf. \cite[proof of Proposition III.1.4]{MR2435846}.
	This proves item (1). Item (2) is  \cite[Reinterpretation III.2.bis.3.a]{MR2435846}
\end{proof}

In the Lemma below, we collect some properties of tails of Hirzebruch-Jung strings for later use.

\begin{lemma}\label{L:tails}
	Let $\F$ be a relatively minimal foliation with pseudoeffective canonical bundle on a smooth
	projective surface $X$. Let $T$ be an irreducible  invariant curve not contained
	in the support of $N$ and let $o_1, \ldots, o_k$ be the orders of Hirzebruch-Jung strings
	intersecting $T$. Then the following assertions hold true.
	\begin{enumerate}	
		\item The intersection of the positive part of the Zariski decomposition
		of $\KF$ with $T$ is given by the formula
		\[
		P \cdot T = \KF \cdot T - \sum_{i=1}^k \frac{1}{o_i} .
		\]
		\item If $\F$ admits a holomorphic first integral $f : U \to \mathbb C$
		defined on a $\F$-invariant neighborhood of  $T$  which vanishes
		along $T$ then the vanishing order along $T$ is a multiple of
		the least common multiple of $o_1, \ldots, o_k$.
	\end{enumerate}
\end{lemma}

\begin{proof}
	Item (1) is \cite[Remark III.1.3.a]{MR2435846}.	
	To verify item (2) let us work locally on a neighborhood $V$ of a Hirzebruch-Jung string
	intersecting $T$. Let $\pi: V \to W$ be the contraction of the Hirzebruch-Jung string we
	are considering and $o$ be its order. Perhaps after restricting $V$ to a smaller
	neighborhood we can assume that $W$ is the quotient of a neighborhood $\tilde V$
	of the origin in $\mathbb C^2$  by a cyclic group generated by
	$\varphi(x,y)= (\xi_o \cdot x, \xi_o^a \cdot y )$ according to Proposition \ref{P:order HJ}.
	We can also assume that the pull-back $\G$ of $\pi_*( \F_{|V})$ to $\tilde V$ is the
	foliation defined by the level sets of the coordinate function $y$.
	The pull-back of $\pi_* (f_{|V})$ to $\tilde V$ is a holomorphic
	function $g$  constant along the leaves of $\G$.
	The $\varphi$-invariance of $g$ implies that $g(x,y)= h(y^o)$ for some one-variable
	holomorphic function $h$. Item (2) follows. 	
\end{proof}

\begin{dfn}
	Let $\F$ be a relatively minimal foliation on a smooth surface $X$ with pseudoeffective
	canonical divisor.
	The nef model of $\F$ is the foliation obtained by contracting
	the negative part of the Zariski decomposition of $\KF$.
\end{dfn}

\subsection{Canonical models}

\begin{dfn}
	A foliation $\F$ on a normal projective surface $X$ is called a canonical model if $\KF$ is nef and
	$\KF\cdot C=0$ implies $C^2 \ge 0$ for every irreducible curve $C\subset X$.
\end{dfn}

\begin{thm}\label{T:McQuillan canonical}\cite[Theorem 3.3.2]{MR2435846}
	Let $\F$ be relatively minimal foliation with $\KF$ pseudoeffective on a smooth surface $X$.
	Then there exists a morphism $\pi:X \to Y$ from $X$ to a normal projective surface $Y$ such that
	$\G=\pi_* \F$ is a canonical model.
	The singular points of $Y$ and the corresponding exceptional fibers of $\pi$ are of one of the following forms.
	\begin{enumerate}
		\item The singular point is a cyclic quotient singularity and the exceptional divisor over it is a  chain of rational curves of self-intersection at most $-2$
		
		\begin{center}
			\begin{tikzpicture}
			\draw (2,0) -- (2.7,0);
			\draw (1,0) -- (2,0);
			\draw (3.3, 0) -- (4,0);
			\draw(4,0) -- (5,0);
			
			\draw[fill=white] (5,0) circle(.1);
			\draw[fill=white] (1,0) circle(.1);
			\draw[fill=white] (2,0) circle(.1);
			\draw[fill=white] (4,0) circle(.1);
			
			\node at (3,0) {$\cdots$};
			\end{tikzpicture}
		\end{center}
		The foliation around the singular point is the quotient of a smooth foliation;
		or the quotient of a canonical foliation singularity on a (germ of) smooth surface;
		\item The singular point is dihedral quotient singularity and the exceptional divisor
		over it has the following dual graph:
		
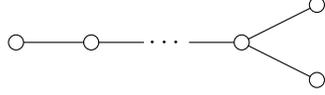
\begin{center}
			\begin{tikzpicture}
			
			\draw (1,0) -- (2.7,0);
			\draw (3.3, 0) -- (4,0);
			\draw (4,0) -- (5,-.5);
			\draw (4,0) -- (5,.5);
			
			\draw[fill=white] (1,0) circle(.1);
			\draw[fill=white] (2,0) circle(.1);
			\draw[fill=white] (4,0) circle(.1);
			\draw[fill=white] (5,-.5) circle(.1);
			\draw[fill=white] (5,.5) circle(.1);
			
			\node at (3,0) {$\cdots$};
			
			\end{tikzpicture}
		\end{center}
		The foliation around the singularity is again the quotient of
		a smooth foliation or of a canonical singularity on a (germ of) smooth surface.
		\item The singular point is an elliptic
		Gorenstein singularity and the exceptional divisor is a cycle of
		smooth rational curves each of self-intersection
		at most $-2$; or a unique nodal rational curve of negative self-intersection
		
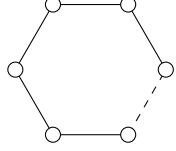
\begin{center}
			\begin{tikzpicture}
			\newdimen\R
			\R=1cm
			\coordinate (center) at (0,0);
			\draw (0:\R)
			\foreach \x in {60,120,...,300} {  -- (\x:\R) };
			\draw[dashed] (300:\R) -- (360:\R);
			
			\draw[fill=white] \foreach \x in {60,120,...,360} { (\x:\R) circle(.1)};
			\end{tikzpicture}		
		\end{center}
		The foliation around the singular point is isomorphic to a cusp of a Hilbert modular
		foliation (cf. \cite[Theorem IV.2.2]{MR2435846}).
		Moreover, the canonical bundle of the foliation on the canonical model is never $\mathbb Q$-Cartier.
	\end{enumerate}
\end{thm}

When compared with the theory for projective surfaces, item (3) of the above Theorem is quite surprising.
The fact that the canonical bundle is never $\mathbb Q$-Cartier is a clear obstruction to the base point
freeness of $|\KF^{\otimes n}|$ and for the finite generation of the canonical algebra
of the foliation. It turns out that this is the only obstruction, cf. \cite[Corollary IV.2.3]{MR2435846}.

\subsection{Kodaira dimension zero}

\begin{thm}\label{T:kod0}\cite[Fact IV.3.3]{MR2435846}
	Let $\F$ be a relatively minimal foliation on a smooth projective surface $X$ with $\nu(\F)=0$.
	Let  $\pi:X \to Z$ be the contraction of the negative part of $\KF$, i.e. $\pi_* \F$ is a nef model for $\F$.
	Then there exists a smooth projective surface $Y$ and a quasi-\'{e}tale cyclic covering  $p:Y \to Z$ of degree $\ii(\F)$
	such that $p^* \mathcal \pi_* \F$ is a foliation with trivial canonical bundle. In particular, $\kod(\F)=0$.
\end{thm}

The resulting surface $Y$ belongs to the following list:
\begin{enumerate}
	\item Product of a hyperbolic curve and an elliptic curve;
	\item Abelian surfaces;
	\item Projective bundle over an elliptic curve;
	\item Rational surface.
\end{enumerate}
Consequently the  klt surface $Z$ has Kodaira dimension  $1$, $0$, or $-\infty$
according to whether $Y$ fits in case (1), (2), or (3)/(4).
One can also determine the possibilities for the index of $\F$. This is done in \cite{MR1913040}.
There it is shown that
\begin{equation}\label{eq.index.0}
\ii(\F) \in \{ 1,2,3,4,5,6,8,10,12\}
\end{equation}
when $\F$ has Kodaira dimension zero.

\subsection{Kodaira dimension one}\label{S:kod0}

The classification of foliations of Kodaira dimension one is essentially due to
Mendes, see \cite[Theorem 3.3.1]{MR1785264}.

\begin{thm}\label{T:kod1}
	Let $\F$ be a relatively minimal foliation on a smooth projective surface $X$.
	Assume that  $\kod(\F)=1$ and let $f:X\to C$ be the Iitaka fibration of $\KF$.
	If $\F$ coincides with the foliation defined by $f$ then $f$ is non-isotrivial elliptic fibration.
	Otherwise $\F$ is completely transverse to a general fiber $F$ of $f$ and we have
	the following possibilities:
	\begin{enumerate}
		\item The genus of $F$ is zero and $\F$ is a Riccati foliation; or
		\item The genus of $F$ is one and $\F$ is a turbulent foliation; or
		\item The genus of $F$ is at least two and $\F$ is an isotrivial fibration
		of genus at least two.
	\end{enumerate}
\end{thm}

\subsection{Non-abundant foliations} The most striking difference between the
birational classification of projective surfaces and the classification of rank one foliations in dimension two
is the existence of foliations having canonical bundle with numerical dimension one and negative Kodaira dimension.
This phenomenon is restricted to a rather special class of foliations as pointed out by the next result.

\begin{thm}\label{T:hilbert}
	Let $\F$ be a relatively minimal foliation on a smooth projective surface $X$.
	If the numerical dimension of $\F$ does not coincide with the Kodaira dimension
	of $\F$ then
	\begin{enumerate}
		\item $\nu(\F)=1$,
		\item $\kod(\F)=-\infty$,
		\item $X$ is the minimal desingularization of the Baily-Borel
		compactification of an
		irreducible quotient of $\mathbb H \times \mathbb H$, and
		\item $\F$ is induced by one of the two natural fibrations
		on $\mathbb H \times \mathbb H$.
	\end{enumerate}
\end{thm}

Arguably this result constitutes the hardest part of the classification of foliations.
The  known  proofs of this result rely heavily on Brunella's plurisubharmonic variation
of the Poincar\'e metric and were obtained by Brunella and McQuillan in a collaborative effort.

In Section \ref{S:classification} we will carry out a classification of foliations
in terms of another birational invariant. This new classification relies heavily on the original classification of surface foliations according to their Kodaira dimension;  nonetheless,
it does not need the full power of it.
In particular, all that we need to know about non-abundant foliations is contained in the following Lemma.

\begin{lemma}\label{L:non-abundant}
	Let $\F$ be a relatively minimal foliation with $\nu(\F)=1$ and $\kod(\F)= -\infty$.
	Then $h^1(X,\mathcal O_X)=0$ and $P \cdot \CNF = P \cdot \KX >0$
	where $P$ is the positive part of the Zariski decomposition of $\KF$.
\end{lemma}

\begin{proof}
	If $h^1(X,\mathcal O_X) = h^0(X, \Omega^1_X) \neq 0$ then
	the restriction of a holomorphic $1$-form to the leaves of $\F$ either vanishes
	identically or gives rise to a non-zero section of $\KF$. Thus if $\kod(\F)=-\infty$
	we obtain that $\F$ factors through the Albanese map of $X$ and is a fibration.
	Hence $\kod(\F)\ge 0$ contrary to our assumptions. Thus $h^1(X,\mathcal O_X)=0$.
	
	Since $h^1(X,\mathcal O_X)=0$ we obtain that $\chi(\mathcal O_X) \ge 1$.
	Let $\mathcal L = \mathcal O_X(mP)$ where $m$ is a sufficiently divisible
	positive integer. By Riemann-Roch,
	\[
	\chi(\mathcal L) = \chi(\mathcal O_X)  + 1/2 ( m^2P^2  - mP \cdot \KX)
	\]
	If $P \cdot \KX <0$ then $\chi(\mathcal L) > 0$.
	Thus $h^0(X,\mathcal L) + h^2(X, \mathcal L) > 0$.
	But if $m$ is sufficiently large then $\KX \otimes \mathcal L^*$ is not pseudoeffective
	and consequently $h^2(X,\mathcal L)= h^0(X, \KX \otimes \mathcal L^*)=0$.
	It follows that $h^0(X, \KF^{\otimes m})=h^0(X,\mathcal L)>0$,
	contradicting $\kod(\F)= -\infty$.
\end{proof}

\section{Effective threshold and  adjoint dimension}\label{S:adjoint}

In this section  we define the effective threshold and the adjoint dimension of a foliation
on a smooth projective surface and prove their birational invariance.

\subsection{Effective threshold}

\begin{dfn}
	Let $\F$ be a foliation with canonical singularities on a smooth projective surface $X$.
	If the canonical bundle of $\F$ is pseudoeffective then
	we define the effective threshold of
	$\F$,  $\eff(\F)$, as the largest
	$\varepsilon \in \mathbb R_{\ge 0} \cup \{ \infty \}$
	such that $\KF + \varepsilon \CNF$ is pseudoeffective.
	If $\KF$ is not pseudoeffective, then we set $\eff(\F) = -\infty$.
\end{dfn}

\begin{example}
	Let $\F$ be a very general foliation on $\mathbb P^2$ of degree $d$.
	It is well known that $\F$ has reduced, and in particular canonical, singularities.
	Recall that the degree of $\F$ is defined as the number of tangencies
	between $\F$ and a general line.
	In this case $\KF=\mathcal O_{\mathbb P^2}(d-1)$ and
	$\CNF = \mathcal O_{\mathbb P^2}(-d-2)$.
	If $d=0$ then $\KF$ is not pseudoeffective. If instead $d\ge 1$ then
	$\KF$ is pseudoeffective and
	\[
	\eff(\F) =  \frac{d-1}{d+2} \, .
	\]
	The reader should notice that $\eff(\F)<1$ for every foliation on $\mathbb P^2$.
\end{example}	

This is by no means a coincidence since
$\KX = \KF + \CNF$ and foliations on a surface  $X$
of negative Kodaira dimension will always have
$\eff(\F)<1$ as $\KX$ is not pseudoeffective.
If instead $X$ has non-negative Kodaira dimension then
$\KX$ is pseudoeffective and consequently $\eff(\F)\ge1$
for every foliation on $X$.

Similarly, one sees that $\eff(\F) = \infty$ if and only if
both $\KF$ and  $\CNF$ are pseudoeffective.
Foliations with pseudoeffective conormal bundle have
recently been classified by Touzet, \cite{2014arXiv14046985}.
They fit in one of the following descriptions:
\begin{enumerate}
	\item after a finite \'{e}tale cover $\F$
	is defined by a closed holomorphic $1$-form; or
	\item there exists a morphism from $X$ to a
	quotient of a polydisc $\mathbb D^m$  by an irreducible lattice
	and $\F$ is the pull-back of one of the $m$ tautological
	foliations on the polydisc. In particular $\F$ is transversely hyperbolic.
\end{enumerate}
Notice that the dimension of the ambient manifold is not necessarily
equal to the dimension of the polydisc.

\begin{remark}\label{R:conversion}
	Using the identity $\KX = \KF + \CNF$  we can write
	\[
	\KF + \varepsilon \CNF =(1 - \varepsilon)( \KF + \frac{\varepsilon}{1-\varepsilon} \KX),
	\]
	when $\varepsilon \neq 1$.
\end{remark}

When $\eff(\F)$ is small, we will often work with
$\KF + \varepsilon \KX$ as that is more  convenient.

\subsection{Adjoint dimension}

\begin{dfn}
	Let $\F$ be a foliation with canonical singularities on a projective surface $X$.
	Consider the pluricanonical  maps
	\[
	\varphi_{m,n} : X \dashrightarrow \mathbb
	P H^0(X, \KF^{\otimes m}\otimes \CNF^{\otimes n})^*
	\]
	for $m\ge 1$, $n\ge 1$.
	The adjoint dimension of $\F$, denoted $\adj(\F)$,
	is the maximal dimension of the image of these maps.
	If $h^0(X, \KF^{\otimes m} \otimes \CNF^{\otimes n})=0$
	for every $m,n \ge 1$  then we set $\adj(\F) = -\infty$.
\end{dfn}

\begin{dfn}\label{D:adj-num.dim}
	Let $\F$ be a foliation with canonical singularities on a projective surface $X$.
	The numerical adjoint dimension of $\F$, $\adjnum(\F)$,
	is equal to $-\infty$ if $\eff(\F)\leq 0$ and
	equal to the maximal numerical dimension of $\KF + \varepsilon \CNF$
	for $\varepsilon \in (0, \eff(\F))$ otherwise.
\end{dfn}

Of course $\adj(\F) \le \adjnum(\F)$.

\subsection{Birational invariance}
The significance of the concepts of effective threshold and of (numerical)
adjoint dimension for the purpose of the birational classification of foliations on
surfaces is assured by the next proposition.

\begin{prop}\label{P:standard}
	Let $(X,\F)$ and $(Y,\G)$ be two birationally equivalent foliations.
	If $\F$ and $\G$ have at worst canonical singularities then
	$\eff(\F)=\eff(\G)$, $\adj(\F)=\adj(\G)$ and $\adjnum(\F)=\adjnum(\G)$.
	Furthermore, $h^0(X, \KF^{\otimes n} \otimes \CNF^{\otimes m})
	=h^0(Y, \KG^{\otimes n} \otimes \CNG^{\otimes m})$
	for every $n,m \ge 0$.
\end{prop}

\begin{proof}
	Since we can choose  a foliation $(Z,\mathcal H)$
	on a smooth projective 	surface $Z$ dominating both $(X,\F)$ and $(Y,\G)$,
	there is no loss of generality in assuming the existence of a
	birational morphism  $\pi:(X,\F)\to(Y,\G)$.
	In fact, we can even assume (and will) that $\pi$ is
	the blow-up of a point $p \in Y$. Let $E$ be the exceptional divisor.
	
	We will first prove that $\eff(\F)=\eff(\G)$.
	First notice that $\KG + \varepsilon \CNG = \pi_*( \KF + \varepsilon \CNF)$.
	Therefore if $\KF + \varepsilon \CNF$ is pseudoeffective
	then the same holds true for $\KG + \varepsilon	\CNG$.
	This shows that $\eff(\G)\ge \eff(\F)$.
	To prove the converse inequality, we will need to use
	that $\G$ has canonical singularities.
	Since $\pi$ is the blow-up of a point by assumption,
	we have that $\KF -\pi^* \KG = a E $ for some $a \in \{ 0 , 1\}$.
	Since  $ \KX - \pi^* K_Y = E$ we also have that
	$\CNF - \pi^* \CNG  = (1-a) E$, and consequently
	\begin{align}\label{Eq:log.pullback.form}
	\KF + \varepsilon \CNF =
	\pi^* (\KG +\varepsilon \CNG) + ( a  + \varepsilon (1- a) ) E.	
	\end{align}
	Therefore, if $\KG +\varepsilon \CNG$ is pseudoeffective
	then the same holds true for $\KF + \varepsilon \CNF$.
	We conclude that $\eff(\G) \le \eff(\F)$ and
	the equality between the effective thresholds follow.
	
	Let us now prove that $\adjnum(\F)=\adjnum(\G)$.
	We have just seen that $\eff(\F)=\eff(\G)$. Then, $\adjnum(\F)= -\infty$
	if and only if $\adjnum(\G)=-\infty$, by definition.
	Let us consider $\varepsilon \in (0, \eff(\F))$ and let
	\begin{eqnarray*}
	\KF + \varepsilon \CNF & = P_{\F, \varepsilon} + N_{\F, \varepsilon},\\
	\KG +\varepsilon \CNG & = P_{\G, \varepsilon} + N_{\G, \varepsilon}
	\end{eqnarray*}
	be the Zariski decompositions of the divisors under scrutiny.
	By \cite[Theorem 2.2]{MR1993748} and \eqref{Eq:log.pullback.form}, it follows that
	for any $\varepsilon \in (0, \eff(\F)), \; P_{\F, \varepsilon}=\pi^\ast P_{\G, \varepsilon}$
	and $N_{\F, \varepsilon}= \pi^\ast N_{\G, \varepsilon}+a  + \varepsilon (1- a)$.
	Since by definition
	$\adjnum(\F)=\max{\nu(P_{\F, \varepsilon}) \; | \; \varepsilon \in (0, \eff(\F))}$
	(and analogously for $\adjnum(\G)$), we have verified the desired equality.
	
	To conclude the proof of the proposition it suffices to verify that
	$h^0(X, \KF^{\otimes n} \otimes \CNF^{\otimes m})=
	h^0(Y, \KG^{\otimes n} \otimes \CNG^{\otimes m})$
	for every $n,m \ge 0$.
	Once these equalities are proved, the equality $\adj(\F)=\adj(\G)$ follows.
	Let us fix $n,m\ge 0$.
	From the isomorphism
	$\KF^{\otimes n} \otimes \CNF^{\otimes m } =
	\pi^*(\KG^{\otimes n} \otimes \CNG^{\otimes m }) \otimes \mathcal O_X (  ( n a + m (1-a) ) E )$
	we deduce the short exact sequence
	\[
	0 \to \pi^*(\KG^{\otimes n} \otimes \CNG^{\otimes m })
	\to \KF^{\otimes n} \otimes \CNF^{\otimes m} \to
	\mathcal O_E (  ( n a + m (1-a) ) E ) \to 0 \, .
	\]
	Since $h^0(E,\mathcal O_E((na+(1-a))E)=0$, we obtain the sought identity.
	
\end{proof}

\subsection{Convention}
For an arbitrary foliation $\F$ on a smooth projective surface
$X$ we define the adjoint dimension, the numerical adjoint dimension and
the effective threshold as the corresponding quantity for any foliation $\G$ with
canonical singularities birationally equivalent to $\F$.

\section{Singularities}\label{S:singularities}

\subsection{Adjoint discrepancy and $\varepsilon$-canonical singularities}

\begin{dfn}
	Let $\F$ be a foliation on $X$ and let $\pi : Y \to X$
	be a birational morphism. Denote by $\G$ the pull-back of $\F$ under $\pi$.
	If $E$ is an exceptional divisor of $\pi$ then the
	adjoint discrepancy of $\F$ along $E$ is the function
	\begin{align*}
	a(\F,E) :  [0,\infty)  & \longrightarrow \mathbb R \\
	t &\longmapsto  \ord_E ( \KG + t \CNG - ( \pi^* \KF + t \pi^* \CNF ) )  \, .
	\end{align*}
\end{dfn}

\begin{dfn}
	Let $\F$ be a foliation on $X$ and $\varepsilon \ge 0$ a real number.
	A point $x \in X$ is $\varepsilon$--canonical if and only
	if the adjoint discrepancy of $\F$ along any divisor $E$ over $x$
	satisfies $a(\F,E)(t) \ge 0$ for every $t\ge \varepsilon$.
	The foliation $\F$ is said to have $\varepsilon$--canonical singularities
	if every point $x \in X$ is $\varepsilon$--canonical.
	The smallest $\varepsilon$ for which $x \in X$ is $\varepsilon$--canonical
	will be called the canonical threshold of $\F$ at $x$.
\end{dfn}

\begin{prop}
	Let $(X,\F)$ and $(Y,\G)$ be two foliations on smooth projective surfaces.
	Assume that	$\F$ and $\G$ are birationally equivalent.
	If both $\F$ and $\G$ have $\varepsilon$-canonical singularities,
	then for any pair of integers $n,m$ satisfying $m/n \ge \varepsilon$ we have that
	\[
	h^0(X,\KF^{\otimes n} \otimes \CNF^{\otimes m} ) =
	h^0(Y,\KG^{\otimes n} \otimes \CNG^{\otimes m} ) \, .
	\]
	In particular, if $\eff(\F)\ge \varepsilon$ then $\eff(\F) = \eff(\G)$.
\end{prop}

\begin{proof}
	The proof is completely analogue to  the proof of Proposition \ref{P:standard}.
\end{proof}

\begin{remark}\label{R:45}
	We point out that $\varepsilon'$-canonical singularities
	are $\varepsilon$-canonical for every $\varepsilon \ge \varepsilon'$.
	In particular, canonical singularities are $\varepsilon$-canonical
	singularities for every $\varepsilon\ge 0$.
	Also note that the canonical threshold of a log canonical
	singularity is at most $1/2$, i.e. log canonical singularities
	are $\varepsilon$-canonical for every $\varepsilon \ge 1/2$.
	This is a straightforward consequence of the simple fact that
	for every divisor $E$ exceptional over $X$ extracted on a smooth birational
	surface $\pi \colon Y \to X$ then $\ord_E (K_Y-\pi^\ast K_X) \in \mathbb{Z}_{> 0}$.
\end{remark}

\begin{notation}
	If $p, q\geq 1$ are relatively prime integers then we will write
	\[
	\frac{p}{q} = [u_0,u_1, \ldots, u_n  ] = \displaystyle \contFrac{u_0 // u_1  // \dots // u_n}
	\]
	for the continued fraction presentation of their quotient.
\end{notation}

If $p,q \ge 1$ are relatively prime positive integers then vector fields of the form $v=p x \px + q y \py$  are tangent to the fibers of rational map $(x,y) \mapsto y^p/x^q$.
As a simple computation shows, see for instance \cite[Appendix 1, Theorem 3]{MR608290} , the minimal reduction of singularities of $v$ coincides with the minimal resolution of indeterminacies of the rational map above by blow-ups on closed points. In particular, there exists a unique irreducible component of the exceptional divisor
which is not invariant by the resulting foliation.

\begin{dfn}\label{D:order}
	Let $p,q\ge 1$ be relatively prime positive integers and consider the germ of foliation
	on $X=(\mathbb C^2,0)$ defined by $v=p x \px + q y \py$.
	Let $\pi :  Y \to X$ be the minimal reduction of singularities of $\F$,
	let $\G$ be the transformed foliation $\pi^* \F$, and   let $E$ be the irreducible
	component of the exceptional divisor which is not $\G$ invariant.
	We will denote the order of $K_Y - \pi^*\KX$ along $E$ by $\varphi(p,q)$ .
\end{dfn}

\begin{lemma}
	Notations as in Definition \ref{D:order}.
	If we write $p/q= [u_0, u_1, \ldots, u_n]$ then the following assertions hold true.
	\begin{enumerate}
		\item $\pi$ is the composition of exactly $\sum_{i=0}^n u_i$ blow-ups; and
		\item the order of $\KY - \pi^*\KX$ along $E$  satisfies
		$\varphi(p,q) \ge \sum_{i=0}^n u_i$.
	\end{enumerate}
\end{lemma}

\begin{proof}
	The key observation is that the reduction of singularities of $v$ follows
	step-by-step Euclid's algorithm for the computation of $\gcd(p,q)$.
	
	Assume that $p\ge q$ and write $p/q$ as a continued fraction $[u_0,u_1, \ldots,u_n]$.
	The proof will be by induction on the number $N=\sum_{i=1}^n u_i$.
	
	If $p=q=1$ then clearly $N=1$ and the result is obvious in this case.
	Assume $p>q$ and consider the blow-up $s \colon Z \to X$ of  the origin with
	exceptional divisor $E_0$.
	Over the exceptional divisor we
	will find two singularities with eigenvalues $(p-q,q)$ and $(p,q-p)$.
	Since we are assuming that $p>q$ then the pair $(p,q-p)$ corresponds
	to a canonical singularity while the pair $(p-q,q)$ corresponds to a
	non-canonical singularity.
	Observe that
	\[
	\frac{p-q}{q}  = [ u_0 -1 , u_1, \ldots, u_n ]  \\
	\]
	Assuming that the result is true for $N-1$ then the first part of the statement follows.
	
	To verify item (2), notice that $\KZ = s^*\KX  + E_0$.
	If $r \colon Y \to Z$ is the minimal desingularization of $s^*\F$
	then by induction hypothesis $\ord_{E} (\KY - r^*\KZ) \ge N-1$.
	Since $\pi=s\circ r$, we can write
	\begin{align*}
	\ord_E (\KY - \pi^* \KX) &= \ord_E ( \KY - r^*(\KZ - E_0 ))= \\
	&= \ord_E(\KY- r^*\KZ) + \ord_E(r^* E_0) \ge N \, .
	\end{align*}
	Then the Lemma follows by induction.
\end{proof}

\begin{remark}
	The inequality in part (2) of the Lemma becomes an equality only
	for singularities with eigenvalues of the form $(1,q)$.
	If $p$ and $q$ are both strictly greater than one,
	at some intermediate step we will be forced to blow-up
	at the intersection of two exceptional divisors
	and one will get a greater order at the end.
	For instance, if $p/q = [u_0,u_1]$ then order of
	$\KY - \pi^* \KX$ along the last exceptional divisor
	is $\varphi(p,q)= (u_1+1)u_0 -1$.
\end{remark}

As a consequence of the above description we are able to characterize
$\varepsilon$-canonical singularities for small values of $\varepsilon >0$.

\begin{prop}\label{P:logcanonical}
	Let $\F$ be a germ of foliation on $(\C^2,0)$.
	If  the canonical threshold of $\F$ at $0$ is
	strictly less than $1/4$ then $0$ is a log-canonical singularity.
\end{prop}

\begin{proof}
	Let $v$ be a generator of $\TF$.
	Assume first that the linear part of $v$ is zero.
	If $\pi: Y \to (\C^2,0)$ is the blow-up of the origin, $\G=\pi^* \F$
	and $E$ is the exceptional divisor then $\KG = \pi^* \KF - a E$, where $a \ge 1$.
	On the other hand $\CNG = \pi^* \CNF + (a+1) E$.
	Therefore, if $\varepsilon < 1/2$ then the origin is not $\varepsilon$-canonical.
	
	Assume now that the linear part of $v$ is non-zero but nilpotent.
	We will use the description of the resolution process of this kind
	of singularities presented in \cite[Chapter 1, proof of Theorem 1]{MR2114696}.
	If we blow-up the origin then we obtain only one singularity
	over the exceptional divisor which 	is invariant by the transformed foliation.
	This new singularity can have zero linear part or non-zero but nilpotent linear part.
	Let us analyze the two possibilities.
	Start with the case where the linear part is zero and
	let $\pi : Y \to (\C^2,0)$ be the composition of the two obvious blow-ups.
	As before we will set $\G=\pi^\ast \F$ and will let $E_1$, $E_2$
	be the two irreducible components of the exceptional divisor of
	$\pi$ with $E_2$ corresponding to the last blow-up.
	Notice that $\KG=\pi^* \KF - a E_2$ for some $a\ge 1$ and
	$\CNG = \pi^*\CNF + E_1 + (a+2)E_2$.
	Hence if $\varepsilon<1/3$ then $0$ is not an $\varepsilon$-canonical singularity.
	Let us now deal with the second possibility.
	If the blow-up of a nilpotent singularity with non-zero
	linear part leads  to another nilpotent singularity with non-zero linear part
	then one further blow-up gives rise to a singularity with trivial linear part.
	Let now $\pi: Y \to (\C^2,0)$ be the composition of the three obvious blow-ups,
	and let $E_1,E_2,E_3$ be the irreducible components of the exceptional
	divisor numbered according to the order of appearance.
	If we set $\G= \pi^* \F$ then $\KG = \pi^* \KF - a E_3$
	for some $a\ge 1$ and $\CNG = \pi^* \CNF + E_1 + 2 E_2 + (a+3) E_3$.
	Thus if $\varepsilon < 1/4$ then $0$ is not a $\varepsilon$-canonical singularity.
	
	Therefore if $\varepsilon<1/4$ then the linear part of $v$ is non-nilpotent
	and we can apply \cite[Fact I.1.8]{MR2435846} to conclude that $0$
	is a log-canonical singularity of $\F$.	
\end{proof}

\begin{cor}\label{C:epsilon open}
	Let $\F$ be a germ of foliation on $(\C^2,0)$ defined by a germ of vector field $v$.
	If $0<\varepsilon<1/4$ then $0$ is a $\varepsilon$-canonical singularity
	of $\F$ if and only if the linear part of $v$ is non-nilpotent and
	one of the following holds:
	\begin{enumerate}
		\item the singularity of $v$ is canonical; or
		\item the singularity of $v$ is not canonical,
		$v$ is analytically conjugated to $ p x \px + qy \py$
		with $p,q$ relatively prime positive integers,
		and $\varphi(p,q)\ge \frac{1- \varepsilon }{\varepsilon}$.
	\end{enumerate}
\end{cor}

\begin{proof}
	Proposition \ref{P:logcanonical} implies that the linear part of $v$ is non-nilpotent.
	If $0$ is not a canonical singularity then by \cite[Fact I.1.9]{MR2435846}
	we know that $v$ is analytically conjugated to   $ p x \px + qy \py$ for suitable
	relatively prime positive integers  $p,q$.
	If $\pi: Y \to X=(\C^2,0)$ is the minimal reduction of singularities of $\F$,
	$E$ denotes the last exceptional divisor and $\G = \pi^* \F$ then $\KG = \pi^* \KF -E$.
	Therefore the adjoint discrepancy of $\F$ along $E$ is (cf. Remark \ref{R:conversion})
	\begin{align*}
	a(\F,E)(t) &=(1-t) \ord_E( \KG + \frac{t}{1-t} \KY - \pi^*(\KF + \frac{t}{1-t} \KX))   =
	\\ &=(1-t) (-1 + \frac{t}{1-t} \varphi(p,q)) .
	\end{align*}
	Since the adjoint discrepancy is clearly non-negative along
	all the other divisors in the minimal resolution it follows that
	$0$ is an $\varepsilon$-canonical singularity if and only if
	$ \varphi(p,q) \ge \frac{1 -\varepsilon}{\varepsilon} $.
\end{proof}

\subsection{Example: log canonical foliations on the projective plane}\label{S:log canonical}

For  a foliation $\F$  on the projective plane with log-canonical singularities
one can easily verify the following assertions.
\begin{enumerate}
	\item If $d=\deg(\F)\ge 4$ then $\eff(\F) = \frac{d-1}{d+2}$.
	\item If $d=\deg(\F) = 3$ then $\eff(\F) = 2/5$ unless $\F$ has radial singularities.
	\item If $d=\deg(\F) = 2$ then $\eff(\F)=1/4$ unless $\F$ has radial
	singularities or dicritical singularities of type $(1,2)$.
\end{enumerate}

One could  try to  pursue a case-by-case analysis
in order to provide an explicit  lower bound for the positive effective
thresholds of foliations of degree two and three with log-canonical singularities.
We will show later in Section \ref{S:variation} that the positive effective
thresholds of foliations varying in an algebraic family do not accumulate at zero.
Unfortunately, our proof is not effective and, a priori, the bound might depend on the family.

\section{Non-isotrivial fibrations}\label{S:Painleve}

\subsection{Producing sections}\label{S:producing}
Our original motivation to introduce and study the adjoint dimension of foliations
lies on our poor understanding of the linear systems $|\KF^{\otimes n}|$. When $\F$ is
a foliation of general type we are not aware of lower bounds on $n$ such that $|\KF^{\otimes n}|$ is not empty. For the linear systems $|\KF^{\otimes n } \otimes \KX^{\otimes m}|$ the situation is considerably better. We can apply the current knowledge on adjoint linear systems to obtain effective bounds on $n,m$ such that $|\KF^{\otimes m}\otimes \KX^{\otimes n}|$ defines a  rational map with two dimensional image.

\begin{prop}\label{P:nonv}
	Let $\F$ be a foliation with canonical singularities on a smooth projective surface.
	If $\kod(\F)=2$  then the linear system  $| \KX + 4 \ii(\F) \KF |$
	defines a rational map with two dimensional image.
\end{prop}

\begin{proof}
	This is an immediate consequence of \cite[Cor. 2]{MR932299}.
\end{proof}

The proposition above is certainly not optimal.
The real question underlying the whole issue here is whether or not
one can provide universal bounds which do not depend on
the index of the foliation. The reader will find a more precise formulation
of this question in Problem \ref{prob:non vanishing}.

\subsection{Bound for the index of hyperbolic fibrations}
In order to use the results above to provide explicit bounds for the degree of leaves of
non-isotrivial hyperbolic fibrations we need to obtain bounds for the index of the foliation.

\begin{lemma}\label{L:lower bound for P}
	Let $\F$ be a relatively minimal foliation on a smooth projective surface $X$.
	Assume $\KF$ is defined by a fibration
	$f:X\to C$ and that the general fiber of $f$ has genus at least two.
	If $T$ is an irreducible curve invariant by $\F$
	which intersects the support of the negative
	part of $\KF$ and it is not contained in it (i.e. $T$ is a tail)
	then one of the following holds:
	\begin{enumerate}
		\item $P \cdot T = 0$ and $T$ intersects exactly two connected
		components of the support of $N$, both of them of order $2$; or
		\item $P \cdot T \ge \frac{1}{42}$.
	\end{enumerate}
\end{lemma}

\begin{proof}
	It follows from  Lemma \ref{L:tails} that
	\begin{equation}\label{E:chi}
	P \cdot T = \KF \cdot  T  - \sum_{i=1}^k  \frac{1}{o_i}
	= -\chi(T) + s + k - \sum_{i=1}^k  \frac{1}{o_i}
	\end{equation}
	where $s$ is the number of singularities of $\F$
	on $T$ which are not contained in the support of $N$,
	\cite[Chapter 2, Prop. 3]{MR2114696}.
	
	Assume  $P\cdot T=0$.
	If $s=0$ then we have the following possibilities for $k$ and $o=(o_1,\ldots, o_k)$:
	$k=3$ and $o=(2, 4, 4)$; or $k=3$ and $o=(3,3,3)$; or $k=3$ and $o=(2,3,6)$; or $k=4$ and $o=(2,2,2,2)$.
	In all cases the whole fiber $F$ containing $T$ is the union of $k$
	Hirzebruch-Jung strings joined by a
	single common tail $T$ and $\chi(F) = \chi_{orb}(\tilde T) = 0$.
	Since $\chi(F)<0$ by assumption, we get that
	$P\cdot T>0$ contradicting our assumption.
	The only remaining possibility is $s=1$, $k=2$ and $o=(2,2)$.
	Item (1) follows.
	
	If $P\cdot T >0$ then it is an elementary and well known fact
	that the lower bound for (\ref{E:chi})
	is equal to $1/42$ and is attained by $s=0$,  $k=3$, and $o=(2,3,7)$.
\end{proof}

\begin{prop}\label{index.prop}
	Let $\F$ be a relatively minimal foliation on a smooth projective surface $X$.
	Assume $\F$ is defined by a fibration $f:X\to C$
	and that the general fiber of $f$ has genus $g \ge 2$.
	Then
	\[
	\ii(\F) \le (42(2g-2))! \, .
	\]
\end{prop}
\begin{proof}
	Let $F = \sum m_i C_i$ be a fiber of $f$ and let $\KF = P+N$
	be the Zariski decomposition of $\KF$.
	If $C_i$ is a tail then, according to Lemma \ref{L:lower bound for P},
	either the Hirzebruch-Jung
	strings intersecting it have order two or $P\cdot C_i \ge 1/42$.
	In the later case we get that
	$m_i \le 42 (2g -2)$ since $P \cdot F = \KF \cdot F = - \chi(F) = 2g-2$.
	Moreover, Lemma \ref{L:tails}
	item (b) implies that the least common multiple of the orders of the
	Hirzebruch-Jung strings intersecting
	$C_i$ divides $m_i \le 42(2g-2)$. The Proposition is then proved.
\end{proof}

\subsection{Boundedness of fibers of non-isotrivial fibrations of a given genus}
Theorem \ref{THM:A} of the Introduction will follow
rather easily from the more general result below.

\begin{thm}\label{T:A}
	Let $\F$ be a foliation with canonical singularities on a smooth projective surface $X$.
	Suppose that $\F$ is a fibration with general fiber $F$ of genus $g$.
	If $\kod(\F)=2$
	then for every nef divisor $H$ we have
	\[
	F \cdot H \le M  ( \KX+ 4 \ii(\F) \KF   ) \cdot H,
	\]
	where $M=M(g)$ satisfies the following inequality
	\[
	M \le 2 ( 4 \ii(\F) +1) ( 2 g -2 ) \le
	\big(4 \big( 42 (2g-2)\big) ! + 1\big) ( 4g -4) \, .
	\]
\end{thm}

\begin{proof}
	When $\F$ is induced by a fibration and $\kod(\F)=2$ then the fibration
	is  non isotrivial and the genus of the general fibre is $\geq 2$, see
	\cite[Theorem 2.1]{MR1177313} and the discussion before that theorem.

	For simplicity we will write $L=\KX + (4 \ii(\F))\KF$.
	Let $F$ denote a general leaf of $\F$.
	If $m\ge 1$ is an integer then
	$mL|_{F} = m(4 \ii(\F) + 1)K_F$.
	On the one hand, Riemann-Roch Theorem implies that
	\[
	h^0(F, \mathcal O_F(mL|_{F})) = m (4\ii(\F) +1) (2g-2) - g +1 .
	\]
	On the other hand, since according to Proposition \ref{P:nonv}
	the linear system $|L|$ defines a rational map
	with two dimensional image, $h^0(X,\mathcal O_X(mL)) \ge \binom{m+2}{2}$.

	If we take $m= 2(4 \ii(\F) + 1)(2g(F)-2)$ then
	\begin{eqnarray*}
	h^0(X, \mathcal O_X(mL-F)) &\geq h^0(X,\mathcal O_X(mL)) - h^0(F, \mathcal O_F(m(4 \ii(\F) + 1)K_F)\ge\\
	&\ge \binom{m+2}{2} - m (4\ii(\F) +1) (2g-2) + g -1 = g.
	\end{eqnarray*}
	In particular, $|mL-F|$ is non-empty.
	As $H$ is an arbitrary nef divisor on $X$ then
	$(mL-F) \cdot H \ge 0$, which concludes the first part of the proof. The inequality in the last part of the statement is then
	just a consequence of Proposition \ref{index.prop}.
\end{proof}

\subsection{Proof of Theorem \ref{THM:A}}
Let $\F$ be a foliation on $\mathbb P^2$ birationally equivalent to a non-isotrivial fibration of genus $g\geq 2$.
Notice that its canonical bundle is isomorphic to $\mathcal O_{\mathbb P^2}(\deg(\F)-1)$.
Let $\pi : X \to \mathbb P^2$ be a birational morphism such
that all the singularities of $\G = \pi^* \F$ are canonical and also
such that on this model the foliation is induced by an actual fibration, that is, there
exists a morphims $g \colon X \to \mathbb{P}^1$ and $\G= \ker dg$.

If we take $H = \pi^* \mathcal O_{\mathbb P^2}(1)$
then the degree of an algebraic leaf $L$ of $\F$ is given by
\[
\deg(L) =  H \cdot \pi^* L = H \cdot \hat L \, ,
\]
where $\hat L$ is the strict transform of $L$. We can thus apply Theorem \ref{T:A}
and Proposition \ref{index.prop} to deduce that
\begin{align*}
\deg (L) &\le  \big(4 \big( 42 (2g-2)\big) ! + 1\big)
( 4g -4) ( \KX+ 4 \ii(\G) \KG   ) \cdot H\\
&=  \big(4 \big( 42 (2g-2)\big) ! + 1\big)
( 4g -4) ( K_{\mathbb{P}^2}+ 4 \ii(\G) \KF) \cdot \pi_\ast H\\
&= \big(4 \big( 42 (2g-2)\big) ! + 1\big)( 4g -4)
( -3 + 4 \big( 42 (2g-2)\big) ! (\deg(\F) -1)) \\ &
\le \Big( 4 \big( 42 (2g-2)\big) !\Big)^2 (4g -4) \deg(\F) \,,
\end{align*}
This concludes the proof of Theorem \ref{THM:A}.
\qed

\subsection{Log canonical foliations on $\mathbb P^2$ of high degree}
The bounds appearing in Theorem \ref{T:A}
are ridiculously large and far from optimal.
Proposition \ref{P:A} below combined with the results
presented in Section \ref{S:variation}
(notably Theorem \ref{T:non accumulation})
indicates that the dependence of $M$ on $g$
in Theorem \ref{T:A} should be at worst linear on $g$.
The results of \cite{charpLu} also indicate the existence of
such linear bounds which are not universal but depend
on the family of foliations in question.

\begin{prop}\label{P:A}
	Let $\F$ be a foliation with canonical singularities on a projective surface $X$.
	Assume that $\F$ is a fibration with general fiber $F$ of geometric genus $g\ge 2$ and
	that  $H^0(X,\KF^{\otimes a}\otimes \CNF^{\otimes b})$
	admits three algebraically independent sections for some $a>0$ and $b\ge 0$.
	Then for every  nef divisor $H$ we have
	\[
	F \cdot H \le 2a(2g-2) ( a\KF+b \CNF   ) \cdot H \, .
	\]
\end{prop}
\begin{proof}
	For simplicity we will write $\mathcal L=a\KF+b\CNF$.
	Let $F$ denote a general leaf of $\F$.
	If $m\ge 1$ is an integer then $mL_{|F} = amK_F$.
	On the one hand, by Riemann-Roch Theorem
	\[
	h^0(F, \mathcal O_F(mL)_{|F}) = m a (2g-2) - g +1 .
	\]
	On the other hand, our assumption on $H^0(X,\mathcal O_X(L))$ implies that
	$h^0(X,\mathcal O_X(mL)) \ge \binom{m+2}{2}$.
	If we take $m= 2a(2g-2)$ then
	\begin{align*}
	h^0(X,\mathcal O_X(mL)) - h^0(F, \mathcal O_F(mL_{|F}))
	& \ge \binom{2a(2g-2)+2}{2} -  2a^2(2g-2)^2 + (g - 1) \\
	& = 6a(g-1) + g > 0.
	\end{align*}
	In particular,
	$|2a(2g-2)L-F|$ is non-empty.
	
	If $H$ is an arbitrary nef divisor on $X$ then
	$(2a(2g-2)L-F) \cdot H \ge 0$, which concludes the proof.
\end{proof}

In the case of foliations of the projective plane with log canonical
singularities and of  degree greater or equal to $5$,
we can actually obtain bounds that are better than linear using a simple variation
of the argument used to prove Proposition \ref{P:A}.

\begin{thm}\label{T:Poincare}
	Let $\F$ be a foliation on $\mathbb P^2$ of degree $d \ge 5$.
	Assume that $\F$ has  log canonical singularities
	and admits a rational first integral with
	general fiber of geometric genus $g\ge 2$.
	If $F$ is a general leaf of $\F$ then
	\[
	\deg(F) \le \Big \lceil\frac{4(2g-2)}{(d-4)^2} \Big\rceil (d-4) \, .
	\]
\end{thm}
\begin{proof}
	Since the singularities of $\F$ are $\varepsilon$-canonical for $\varepsilon=1/2$
	(see Remark \ref{R:45}) we have that the dimension of the vector spaces
	$H^0(\mathbb P^2, \KF^{\otimes 2m} \otimes \CNF^{\otimes m}), \; m> 0$
	is unaltered after replacing $\F$ by a model with at worst canonical singularities.
	
	Let $F$ be a general fiber of the rational first integral of $\F$
	and consider the real-valued function
	\[
	f(m) = \binom{m(d-4) + 2}{2} - 2m(2g-2) + g - 1.
	\]
	Its values on positive integers correspond to the difference
	$h^0(\mathbb P^2, \KF^{\otimes 2 m } \otimes \CNF^{\otimes m})
	- h^0(\tilde F,K_{\tilde F}^{\otimes 2m})$,
	where $\tilde F$ is the normalization of $F$.
	Since $f(4(2g-2)/(d-4)^2)=(dg + 8g -12)/(d-4)$ which is clearly positive
	and moreover the derivative of $f$ satisfies $f'(4(2g-2)/(d-4)^2)= (3/2)d + 4g -10>0$,
	it follows that if $m$ is the smallest integer greater than $4(2g-2)/(d-4)^2$ then
	there exists a section of
	$\KF^{\otimes 2m} \otimes \CNF^{\otimes m}\simeq \mathcal O_{\mathbb P^2}(m(d-4))$
	vanishing identically on $F$. The Theorem follows.
\end{proof}

As already mentioned in the Introduction,
this Theorem \ref{T:Poincare} refines a classical
result of Poincar\'e, see \cite[pages 169 and 176]{PoincarePalermoI}
and \cite[Chapter 7, Corollary 14]{MR2029287}.

\section{Classification via adjoint dimension}\label{S:classification}

In this section we apply the results recalled in Section \ref{S:Kodaira}
to obtain a classification of foliations on surfaces according to their adjoint dimension.

\subsection{$\KX$-negative extremal rays}

Recall that for a smooth projective surface $X$ the $\KX$-negative extremal rays
are spanned by numerical classes of rational curves of self-intersection either
$-1, 0$ or $1$. The first case corresponds to the exceptional divisor of the blow-up
of a smooth point, the second to a smooth fiber of a $\mathbb{P}^1$-bundle,
while the last one is just the class of a line in $\mathbb{P}^2$.

\begin{lemma}\label{L:chi}
	Let $\F$ be a relatively minimal foliation with
	pseudoeffective $\KF$ on a smooth projective surface $X$,
	and let $\KF =P + N$ be the Zariski decomposition of $\KF$.
	Assume there exists a $\KX$--negative extremal curve
	$C \subset X$ and $P\cdot C = 0$.
	Then the Kodaira dimension of $\F$ is either $0$ or $1$.
	Moreover, if $\kod(\F) = 1$, then the image of $C$ in the canonical model $\pi : X \to Z$
	of $\F$ is proportional to $\pi_* \KF$.
\end{lemma}

\begin{proof}
	If $C$ is an extremal ray with $C^2 \ge 1$
	then Hodge index theorem implies that $P$ is numerically zero.
	Theorem \ref{T:kod0} implies $\kod(\F)= 0$.
	
	If instead $C^2=0$ then $P$ is numerically equivalent
	to a non-negative multiple of $C$ and
	we deduce that either $\nu(\F)=0$ or $\nu(\F)=1$.
	The case $\nu(\F)=0$ follows as before. If $\nu(\F) =1$ and
	since $P$ is numerically equivalent to an effective divisor,
	we can apply Theorem  \ref{T:hilbert} and Lemma \ref{L:non-abundant}
	to deduce that  $\kod(\F)=1$.
	
	From now on assume that $C^2=-1$ and
	let $\pi : X \to Y$ be the contraction
	of $\F$ into its canonical model.
	If $C$ is not contracted by $\pi$
	then write $\pi^* \pi_* C = C + \sum a_i E_i$ where
	$a_i >0$ and the $E_i$ are $\pi$-exceptional divisors.
	Thus $\pi_* P \cdot \pi_* C = P \cdot  \pi^* \pi_* C = P \cdot C$
	since $P$ is the pull-back of a nef divisor from $Y$ and
	hence $\pi$-exceptional curves intersect $P$ trivially.
	As we are assuming $P\cdot C= 0$
	we deduce from Hodge index Theorem that either $P$ is numerically trivial,
	or that $\pi_* C ^2 = 0$ and $\pi_*P$ is numerically equivalent
	to a positive multiple of $\pi_* C$.
	Hence $\nu(\F) \in \{0,1\}$.
	As before, we obtain that in both cases $\nu(\F) = \kod(\F)$.
	
	Suppose now that $C$ is contracted by $\pi$.
	In this case $C$ is $\F$-invariant according to Theorem
	\ref{T:McQuillan canonical}.
	Since $C^2 = -1$ and $\F$ is relatively minimal we have that $Z(\F,C) \ge 3$.
	Notice that $\KF \cdot C = -2 + Z(\F,C)$ and, as we are assuming $P \cdot C=0$,
	according to Lemma \ref{L:tails} we also have that
	$\KF \cdot C = \sum_{i=1}^k 1/o_i$ where
	$o_i$ are the orders of the Hirzebruch-Jung strings intersecting $C$.
	Then we must have $k=2$ and $o_1=o_2=2$;
	or $k=3$ and $(o_1,o_2,o_3) \in \{ (2,3,6) , (3,3,3) \}$;
	or $k=4$ and $(o_1,o_2,o_3,o_4)=(2,2,2,2)$.
	If we contract the Hirzebruch-Jung strings intersecting $C$,
	we obtain that the direct image of $C$ has self-intersection
	$\ge 0$, cf. \cite[Remark III.2.2]{MR2435846}.
	Thus $C$ cannot be contracted by $\pi$ contrary to our assumption.
\end{proof}

\subsection{Kodaira dimension zero}
\begin{lemma}\label{L:N-delta}
	Let $\F$ be a relatively minimal foliation with pseudoeffective
	$\KF$ on a smooth projective surface $X$.
	If $\pi:X \to Z$ is the contraction of the
	negative part of $\KF$ (i.e. $\pi_* \F$ is a nef model of $\F$)
	and we write  $\KX  + \Delta= \pi^* \KZ$ then
	$\ii(\F) N-\Delta$ is effective.
\end{lemma}

\begin{proof}
	If $E_1, \ldots, E_k$ are the exceptional divisors of $\pi$ then
	$\Delta$ is defined by the relations
	\[
	\Delta\cdot E_i = -\KX\cdot E_i = 2 + E_i^2  \, .
	\]
	Notice that $ 2 + E_i^2 \le 0$ for every $i$,
	while $2+E_i^2 \ge (E_1 + \cdots + E_k)\cdot E_i$
	for every $i$ and the latter inequality is strict
	when $E_i$ is either a handle or a tail in a Hirzebruch-Jung string.
	Therefore by \cite[Corollary 4.2]{MR1658959}
	the coefficients of  $\Delta$ lie in $[ 0,1)$.
	Since $N$ is effective the lemma follows.
\end{proof}

\begin{prop}\label{P:kod0}
	Let $\F$ be a relatively minimal foliation
	of Kodaira dimension zero on a smooth projective surface $X$.
	If $\pi:X \to Z$ is the contraction of the
	negative part of the Zariski decomposition of $\KF$
	and $(X, \Delta)$ is the pair
	satisfying $\KX + \Delta = \pi^* \KZ$ then
	the adjoint dimension and the numerical
	adjoint dimension of $\F$ coincide with
	the Kodaira dimension of $(X,\Delta)$.
	Moreover, when $\adj(\F)\ge 0$ then
	$\eff(\F) \ge \frac{1}{\ii(\F)+1} \ge \frac{1}{13}$.
\end{prop}

\begin{proof}
	Let $\KF=P+N$ be the Zariski decomposition of $\KF$.
	Since we are assuming that $\F$ has Kodaira dimension zero we have that $P\sim_\mathbb{Q} 0$.
	Let $\pi: X \to Z$ be the contraction of the support of
	$N$ and notice that we can write
	\[
	\KF + \varepsilon \KX = \varepsilon \pi^* \KZ + ( N - \varepsilon \Delta) .
	\]
	Assume that $\varepsilon$ is rational and satisfies $\varepsilon < 1/\ii(\F)$.
	Lemma \ref{L:N-delta} implies  that $( N - \varepsilon \Delta)$ is effective.
	Hence for any $k$ sufficiently divisible,
	$h^0(X,k ( \varepsilon \pi^* \KZ + ( N - \varepsilon \Delta) ))
	\ge h^0(X,k \varepsilon\pi^* \KZ)= h^0(Z, k \varepsilon \KZ)$.
	Since  every irreducible component $E$ of the support of
	$( N - \varepsilon \Delta)$ is $\pi$-exceptional  we also have the opposite inequality.
	This shows that the Kodaira dimension of $Z$ is equal to the adjoint dimension of $\F$.
	
	To verify that the adjoint dimension and the numerical
	adjoint dimension of $\F$  coincide first observe that
	every irreducible component $E$ of of $N - \varepsilon \Delta$
	satisfies $\pi^* \KZ \cdot E =0$.
	Therefore the numerical dimension of $\KF + \varepsilon \KX$
	coincides with the numerical dimension of $\KZ$.
	As the numerical dimension of $\KZ$ and the Kodaira dimension
	of $(X,\Delta)$ coincide, the Proposition follows.
	
	Finally, the last part of the statement follows from the fact that for
	a relatively minimal foliation $\F$ of Kodaira dimension $0, \; \ii(\F)\leq 12$,
	see \eqref{eq.index.0}.
\end{proof}

\subsection{Kodaira dimension one}

\begin{prop}\label{P:kod1}
	Let $\F$ be a relatively minimal foliation of Kodaira dimension one
	on a smooth projective surface $X$.
	Let $g$ be the genus of a general fiber of
	the Iitaka fibration of $\F$.
	If $g=0$ then $\adj(\F) = \adjnum(\F) = -\infty$.
	Otherwise
	\[
	\adj(\F) = \adjnum(\F) = \min\{ g,2\} \,
	\text{ and } \eff(\F) \ge \frac{1}{4\ii(\F)  + 1}.
	\]
\end{prop}

\begin{proof}
	Let $f: X \to B$ be the Iitaka fibration of $\F$.
	Assume first that $g=0$. 	
	Then for a general fiber $F$ of $f$ we have that $\KF \cdot F= 0 $ and
	$\KX \cdot F = -2$. Hence $\KF + \varepsilon \KX$ is not
	pseudoeffective for every $\varepsilon >0$. It follows
	that $\adj(\F)=\adjnum(\F)=-\infty$.
	
	Assume now that $g\ge 1$.
	Let $\KF = P+N$  be the Zariski decomposition of $\KF$ and
	let $\pi:X \to Z$ be the contraction of the negative part of $\KF$.
	Denote by $\G$ the direct image of $\F$.
	We claim that $4\ii(\F) \KG + \KZ$ is nef.
	Suppose not, and let $D$ be an effective divisor
	such that $(4\ii(\F) \KG + \KZ)\cdot D < 0$.
	By the Cone Theorem we can numerically decompose
	$D$ as a sum $\sum a_i C_i + R$ where $R$ is a
	pseudoeffective divisor and satisfies $K_Z\cdot R \ge 0$;
	$C_i$ are $\KZ$--negative extremal rays
	satisfying $0<-\KZ \cdot C_i \le 4$  and  $a_i \in \mathbb R_{>0}$.
	Therefore, there exists a $\KZ$--negative extremal
	ray $C$ such that  $(4\ii(\F) \KG + \KZ)\cdot C < 0$.
	If $\KG \cdot C =0$ then Lemma \ref{L:chi} implies that
	$C$ is numerically proportional to $\KG$. Infact, denoting by $\tilde{C}$ the strict
	transform of $C$ on $X$, we have that $P \cdot \tilde{C} = \pi^ \KG \cdot \tilde{C}=0$.
	Consequently $C$ is proportional to a general fiber of
	$f \circ \pi^{-1}$ and must intersect $\KZ$ non-negatively.
	Thus $\KG\cdot C >0$.
	Since $\ii(F) \KG$ is Cartier we deduce that $4\ii(\F) \KG \cdot C \ge 4$.
	It follows that also in this case  $(4\ii(\F) \KG + \KZ)\cdot C \ge 0$.
	We conclude that $4 \ii(\F) \KG + \KZ$ is nef.
	Consequently we obtain that
	\begin{equation}\label{E:ZD}
	\KF + \frac{1}{4\ii(\F)} \KX =
	\pi^* \left( \KG + \frac{1}{4\ii(\F)} \KZ\right) +
	\left( N - \frac{1}{4\ii(\F)} \Delta \right)
	\end{equation}
	where $\Delta$ is defined by $\KX + \Delta= \pi^* \KZ$.
	Since the singularities of $Z$ are klt, it follows that
	$N - \frac{1}{4\ii(\F)} \Delta$ is effective and that
	$ \KF + \frac{1}{4\ii(\F)} \KX $ is pseudoeffective.
	Thus $\eff(\F) \ge \frac{1}{4 \ii(\F) + 1}$.
	
	It remains to determine the adjoint dimension of $\F$.
	For that, notice that (\ref{E:ZD}) is the Zariski decomposition
	of  $\KF + \frac{1}{4\ii(\F)} \KX$.
	When $g = 1$, since $\KX$ is trivial when restricted to the
	general fiber of $f$ it follows that the positive part
	$\pi^* \left( \KG + \frac{1}{4\ii(\F)} \KZ\right)$ is
	numerically proportional to a general fiber and also that
	there exists an a effective $\mathbb Q$-divisor
	$D$ on $B$ such that $\pi^* \left( \KG + \frac{1}{4\ii(\F)} \KZ\right) = f^* B$.
	Hence $\adjnum(\F) = \adj(\F) = 1$.
	
	To prove the claim for  $g\ge 2$ it suffices to verify that
	$\pi^* \left( \KG + \varepsilon \KZ\right) ^2 > 0$
	for $\varepsilon$ sufficiently small.
	If this were not the case then $\KG \cdot \KZ =0 $ and
	$\KZ \cdot \KZ =0$.
	The Hodge index theorem would imply that $\pi^* \KZ$ is
	proportional to a general fiber $f$.
	But this is not possible since
	$\pi^* \KZ  \cdot F = 2g - 2>0$ for any fiber $F$ of $f$.
\end{proof}

\subsection{Kodaira dimension two and non-abundant foliations}

\begin{lemma}\label{L:Eff}
	Let $\F$ be a relatively minimal foliation
	with canonical singularities
	which is not a fibration by rational curves.
	Let $\KF=P+N$ be the Zariski decomposition of $\KF$.
	If $\kod(\F) \notin \{ 0, 1\}$ then
	$P + \frac{1}{3\ii(\F)} \KX$ is  nef.
\end{lemma}

\begin{proof}
	Aiming at a contradiction,
	let $C$ be a curve  such that $(P+1/3\ii(\F) \KX)\cdot C < 0$.
	As in the proof of Proposition \ref{P:kod0}
	we can assume that $C$ is a $\KX$-negative extremal curve
	and therefore  $\KX \cdot C \in \{ -3,-2,-1\}$.
	By Lemma \ref{L:chi}, $P\cdot C >0$. Hence
	\[
	-\KX \cdot C_i > 3\ii(\F) (P \cdot C_i) \ge 3  \,
	\]
	gives the sought contradiction.
\end{proof}

\begin{prop}\label{P:kod2}
	Let $\F$ be a relatively minimal foliation
	with canonical singularities and pseudoeffective canonical bundle.
	If $\kod(\F) \notin \{ 0, 1\}$ then $\adjnum(\F)= \adj(\F)=2$.
\end{prop}

\begin{proof}
	Let $\KF = P + N$ be the Zariski decomposition of $\KF$.
	Since $\kod(\F) \neq0$ we have that $\nu(\F)\ge 1$.
	Lemma \ref{L:Eff} implies that
	$P + \varepsilon \KX$ is nef for $0<\varepsilon$ sufficiently small.
	
	Assume by contradiction that $\F$ is not of adjoint general type.
	Then $(P + \varepsilon \KX)^2$ must vanish identically. In fact, if
	$(P + \varepsilon \KX)^2$ does not vanish identically, then
	for $0<\varepsilon$ sufficiently small $(P + \varepsilon \KX)^2>0$
	by the nefness of $P + \varepsilon \KX$. That
	implies that $P + \varepsilon \KX$ is big $0<\varepsilon$ sufficiently small
	and the same holds true for $\KF+\varepsilon \KX$, as $N \ge 0$,
	reaching a contradiction.
	From this observation, it follows that $P^2 = P \cdot \KX = \KX ^2 =0$.
	Lemma \ref{L:non-abundant} implies that $\kod(\F)\ge 0$.
	But this is not possible by the assumptions in the statement of the proposition.
	Hence we reach the desired contradiction and the result follows.
\end{proof}

\subsection{Characterization of rational fibrations (Proof of Theorem \ref{THM:C}) }
One immediate consequence of the classification of foliations
according to their adjoint dimension is the characterization
of rational fibrations stated in the Introduction as Theorem \ref{THM:C}.

\begin{thm}\label{T:rational}
	Let $\F$ be a foliation with canonical singularities
	on a smooth projective surface $X$.
	Then $\F$ is a rational fibration if and only if
	$h^0(X,\KF^{\otimes m} \otimes \CNF^{\otimes n})=0$
	for every $m>0$ and every $n\ge 0$.
\end{thm}

\begin{proof}
	If $\adj(\F)\ge 0$
	then $h^0(X,\KF^{\otimes m} \otimes \CNF^{\otimes n})\neq 0$ for some $m,n>0$
	by definition.
	If instead $\adj(\F)=-\infty$ and $\F$ is not a fibration by rational curves
	then by Theorem \ref{T:non.pseff.fibr},
	$\KF$ must pseudoeffective and $\kod(\F)=-\infty$.
	Moreover, we can assume that $\F$ is
	relatively minimal by Proposition \ref{P:standard}.
	But then Proposition \ref{P:kod2} gives a contradiction, as
	it would imply that $\adj(\F)=2$.
\end{proof}

For foliations on smooth surfaces of Kodaira dimension $0$ or $1$,
$h^0(X,\KF^{\otimes n}) > 0$ for some $n$ between $1$ and $12$,
see \cite{MR2177633} and \cite[Section 4]{2016arXiv160405276C}.
It is a simple matter to obtain effective non-vanishing of
$h^0(X,\KF^{\otimes n} \otimes \CNF^{\otimes m})$
for foliations $\F$ of adjoint general type as functions of their index $\ii(\F)$.
This is what we did in the proof of \ref{P:nonv} when $\kappa(\F)=2$.
The real question here is if one can do that
regardless of the index of the the foliation.

\begin{problem}\label{prob:non vanishing}
	Find universal bounds on $(n,m) \in \mathbb{Z}_{>0} \times \mathbb{Z}_{>0}$
	in order to ensure  the non-vanishing of
	$h^0(X,\KF^{\otimes m}\otimes \CNF^{\otimes n})$
	for foliations of adjoint general type.
\end{problem}

For bounded families of foliations,
the results of Section \ref{S:variation}
imply the existence of bounds depending on the family.

\section{Variation in moduli}\label{S:variation}

\subsection{Families of foliations}
We start by spelling out the definition of family
of foliated surfaces.

\begin{dfn}\label{dfn:family}
	Let $\pi:\XX \to T$ be a family of smooth projective surfaces,
	i.e. $\XX$  and $T$ are irreducible complex manifolds
	and $\pi$ is a proper submersion with projective surfaces as fibers.
	A family of foliations parametrized by $T$ is a foliation $\FF$ of
	dimension one on $\XX$ which is everywhere tangent to the fibers of $\pi$.
	If $\XX, T, \pi$ and $\FF$ are all algebraic then we say that $\FF$
	is an algebraic family of foliations.
\end{dfn}
For the sake of notation, we will indicate with $\XX_t$ the (schematic) fiber of $\pi$
over a point $t \in T$ and with $\FF_t, \; \FF_{|\XX_t}$.

Notice that in the definition above we do not impose any condition
on the nature of singularities of $\FF$, contrary to what is done
in \cite{MR1929324} where the singularities are supposed to be reduced.
Also when the dimension of $T$ is at least two
it may happen that some fibers of $\pi$ are contained in the singular
set of $\FF$.

\begin{remark}\label{R:sing}
Although no assumption is made on the nature of the singularities,
the singular set of $\FF$ has codimension at least two as the singular set
of any foliation on a smooth manifold. Then
there exists a non empty Zariski open subset $U \subset T$ such that $\sing(\FF_t)
= \sing(\FF) \cap \XX_t$ for every $t \in U$. In particular,
$K_{\FF_t} = (K_{\FF})_{|\XX_t}$, for every $t \in U$. Note also
that for every $t \in U$, we have the equality
${\det N^*_{\FF}}_{|\XX_t} = N^*_{{\FF}_t}$.
\end{remark}

It is useful to think of an algebraic family of foliations parametrized
by $T$ as a foliation defined over the function field $\mathbb C(T)$.
Algebraic properties of a very general member $\FF_t$ of the family --
like existence of invariant algebraic curves, rational first integrals,
transversely projective structures -- are displayed already when one
considers the foliation as defined over $\mathbb C(T)$.
Also the Kodaira dimension (resp. the adjoint dimension) of the foliation defined
over $\mathbb C(T)$ coincides with the Kodaira dimension (adjoint dimension)
of a very general member of the family.

\subsection{Partial reduction of singularities for families}
When the singularities of a family of foliations $\F$ on smooth surfaces
are reduced, Brunella \cite{MR1929324} showed that the Kodaira dimension of the
foliations is constant in the family. Later, Cascini and Floris
proved that for families of foliations with reduced singularities and for $m$ sufficiently large, the m-th plurigenera 
$h^0(\mathcal{X}_t, m\KF_t)$
is constant. They also presented examples  of families of foliations with reduced singularities
which have non constant m-th plurigenera for small values of $m$, see \cite[Section 3.4]{2015arXiv150200817C}. 

When we do not restrict to  families of foliations with
reduced singularities, then one of the sources of difficulties in applying
birational techniques similar to those of \cite{2015arXiv150200817C}
to understand the behavior of the plurigenera comes from the fact that canonical singularities
are not stable in the Zariski topology: the set of foliations with
at worst canonical singularities can fail to be Zariski open as the
family of foliations on $\mathbb C^2$ parametrized by $\mathbb C$
and defined by $xdy - t y dx$ shows.
In this family the singularity at the origin is canonical if and only if
$t \notin \mathbb Q_+$. Thus a very general foliation in the family has
canonical singularities, but the set of foliations with non-canonical
singularities is Zariski dense. This unpleasant situation can be avoided
if instead one considers $\varepsilon$-canonical singularities for $\varepsilon>0$.

\begin{lemma}\label{L:epsilon open}
	Let $\mathscr F$ be an algebraic  family of foliations parametrized
	by an algebraic variety $T$.
	If $0< \varepsilon< 1/4$ then the subset of $T$ corresponding to
	foliations with isolated and $\varepsilon$-canonical singularities
	is a Zariski open subset of $T$.
\end{lemma}

\begin{proof}
	This is a simple consequence of Corollary \ref{C:epsilon open}.
	If a singularity is not $\varepsilon$-canonical,
	$0<\varepsilon < 1/4$, then either its linear part is
	nilpotent or the singularity is formally equivalent to one of the finitely many
	singularities of the form  $ p x \px + qy \py$ with $p,q$ relatively prime
	positive integers satisfying $\varphi(p,q) < \frac{\varepsilon}{1-\varepsilon}$
	(see Definition \ref{D:order} for the meaning of $\varphi$).
	Since both conditions are clearly closed the lemma follows.
\end{proof}

\begin{prop}\label{P:Seidenberg}
	Given an algebraic family of foliation $\FF$ parametrized
	by an algebraic variety $T$ and a real number $\varepsilon >0$, there
	exists a non-empty Zariski open subset $U \subset T$ and a family
	of foliations $\GG$ on $\YY \to U$	obtained from $\FF_{|U}$
	by a finite composition of blow-ups over (multi)-sections
	such that for every closed point $t \in U$,
	the foliation $\GG_{t}$ has at worst $\varepsilon$-canonical singularities.
\end{prop}

\begin{proof}
    First replace $T$ by a Zariski open subset in such a way that the singular scheme
    of $\FF$ becomes flat over $T$. In particular, every irreducible component of the singular
    set  of $\FF$ is a multi-section of the projection to $T$.

    We will say that an irreducible component $\Sigma$ of the singular
    set of $\FF$ is  reduced when, for a very general $t \in T$,  the points in $\Sigma_t$ are
    reduced singularities for $\FF_t$.

    If we interpret $\FF$ as a foliation $\FF_{\overline{\mathbb C(T)}}$ over $\overline{\mathbb C(T)}$, the Zariski closure of
    the function field of $T$, then the reduced/non-reduced singularities of $\FF_{\overline{\mathbb C(T)}}$ correspond, respectively, to
   irreducible components of $\sing(\FF)$ having reduced/non-reduced singularities at the very general point.

    The proof of Seidenberg's theorem presented in  \cite[Appendix 1]{MR608290} and \cite{MR2114696}  works for arbitrary algebraically closed fields
    of characteristic zero.  It consists of showing that the iterated blow-ups of non-reduced points will eventually
    lead to foliations with only reduced singularities. First one aims to obtain foliations with all singularities having
    non-trivial linear part.   To achieve this one controls the multiplicity of the singularities
    of the foliation by means of Van den Essen formula which is established in \cite[pages 52-53]{MR548142} (see also   \cite[Appendix 1, formulas (2.1) and (2.2)]{MR608290})
    for arbitrary algebraically closed fields of characteristic zero. Then to go from non-trivial linear part
    to non-nilpotent linear part, one carries out explicit algebraic computation computations, for details see
    \cite[Appendix 1, Theorem 2]{MR608290}. Finally, to go from non-nilpotent linear part to quotient of eigenvalues
    not belonging to $\mathbb Q_{>0}$ one relies on Euclidean division algorithm \cite[Appendix 1, Theorem 3]{MR608290}.

    Therefore blowing-up non-reduced multi-sections, restricting $T$ to  suitable Zariski open subsets, and  repeating if necessary,
    will lead to the sought a family of foliations over a Zariski open subset of $T$ with all irreducible components of its singular
    set reducible. We apply Lemma \ref{L:epsilon open} to conclude.	
\end{proof}

\subsection{Families of foliations of negative adjoint dimension}
Foliations of negative adjoint dimension also behave better in
families compared to foliations of negative Kodaira dimension.

\begin{lemma}\label{L:negative open}
	Let $(\pi: \XX \to T, \FF)$ be an algebraic family of foliations.
	If for a very general closed point $t_0 \in T$ the foliation
	$\FF_{t_0}$ is reduced and has negative adjoint dimension
	then  there exists a non-empty Zariski open subset $U \subset T$ such
	that for every closed point $t \in U$ the foliation $\FF_t$
	has negative adjoint dimension.
\end{lemma}

\begin{proof}
	Assume first that for  a very general point $t \in T$ the foliation
	$\FF_t$ has Kodaira dimension one. Since the adjoint dimension is negative,
	$\FF_t$ must be a Riccati foliation.
	It follows from \cite[Proposition 4.3]{2016arXiv160405276C} that for some
	$n \le 42$ the linear system $|K_{\FF_t}^{\otimes n}|$ is non-empty and
	defines the reference rational fibration.
	Moreover, the general fiber of the reference fibration intersects
	$K_{\FF_t}$ trivially. After replacing $T$ by a Zariski open subset we can assume that $\sing(\FF_t) = \sing(\FF) \cap \XX_t$, see Remark \ref{R:sing}.
    We can also assume that  $\pi_* K_{\FF}^{\otimes n}$ is locally free. Consider the rational map defined by pluricanonical sections
    \[
        \XX \dashrightarrow \mathbb P ( (\pi_* K_{\FF}^{\otimes n})^* ) .
    \]
    We can further restrict $T$ in order to get an actual morphism.  Let $P : \XX \to B \subset \mathbb P ( (\pi_* K_{\FF}^{\otimes n})^* )$
    be such morphism. We do not claim
    that $P$ has irreducible fibers, but since $\FF_t$ is a Riccati foliation of Kodaira dimension one for a very general $t$,
    it follows that  very general fibers of $P$ are smooth
    curves with rational connected components, completely transverse to $\FF_t$, and  intersecting  $K_{\FF}$ trivially.
    Hence, using for instance Ehresmann's fibration theorem, we can conclude that the same holds true
    for fibers of $F$ over a Zariski open subset $V$ of $B$. Let $U \subset T$ be a Zariski open subset contained in the image of $V$
    under the natural projection $B\to T$.
    We can apply \cite[Proposition 4.1]{MR2114696} to deduce that for every $t \in U$
	the foliation $\FF_t$ is a Riccati foliation and as such has
	negative adjoint dimension.
	
	Assume now that for a very general point $t \in T$ the foliation $\FF_t$
	has Kodaira dimension zero. We will make use of Lefschetz principle to deal
    with this case. Since the family $\FF$ is defined by finitely many equations, we can consider a finitely
    generated extension $K$ of $\mathbb Q$ over which everything in sight is defined. Embed $K$ into $\mathbb C$  and apply Theorem \ref{T:kod0}.
	We deduce that  after restricting $T$ to a Zariski open subset $U$ and base changing the family $\FF$ through an \'etale covering $V\to U$
    we obtain that the resulting family $\XX' \to U$ is birationally equivalent to a finite quotient of a smooth family of foliations $\mathscr G$
    on $\mathscr{Z} \to V$ defined by global holomorphic vector fields. Since we are assuming that for a very general $t\in T$ the foliation has negative
    adjoint dimension it follows that the very general fiber of $\mathscr{Z} \to V$ is a surface of negative Kodaira dimension and the corresponding
    foliation is a Riccati foliation. It follows that for every $t \in U$, $\FF_t$ has negative adjoint dimension.
	
	Finally, let us assume that  a very general $t \in T$ the foliation $\FF_t$ is a rational fibration. Let $U \subset T$ be
	a Zariski open subset over which every irreducible component of the singular scheme of $\FF$ is flat. As we are assuming that $\FF_t$ is a foliation with reduced singularities tangent to a rational fibration for a very general $t \in T$,  it follows that for  a very general  $t \in U$ the singularities of $\FF_t$ have invertible linear part and  quotient of eigenvalues in $\mathbb Q_{<0}$. Flatness of the singular scheme of $\FF$ over $U$ guarantees that the singularities of $\FF_t$ have invertible linear part for every $t \in U$. Continuity of the quotient of eigenvalues ensures that they are constant  functions of $t \in U$. Thus, for every $t \in U$, the foliations $\FF_t$ have reduced singularities. It follows from \cite[Lemma 1]{MR1929324} that $\FF_t$ is a rational fibration for every $t\in U$.
\end{proof}

\subsection{Boundedness of the effective threshold in families}
We have now all the ingredients to prove the result mentioned at
the end of Section \ref{S:log canonical}.

\begin{thm}\label{T:non accumulation}
	Let $(\pi: \XX \to T, \FF)$ be an algebraic family of foliations.
	Then there exists $\delta >0$ such that, for every $t \in T$,  the following
	holds true: $\adj(\FF_t)=-\infty$ or $\eff(\FF_t)\ge \delta$. In other words,
	if $\eff(\FF_t) < \delta$ then $\adj(\FF_t)=-\infty$.
\end{thm}

\begin{proof}
    The proof is by Noetherian induction.
    Therefore it suffices to show that the  result is valid over a Zariski open subset of $T$.

    In particular, after replacing $T$ by a Zariski open subset, we can assume     that $\sing(\FF_t) = \sing(\FF) \cap \XX_{t}$ for every $t \in T$.

	Proposition \ref{P:Seidenberg} guarantees that there is no loss of generality in assuming that
	$\FF_t$ has canonical singularities for a very general $t \in T$.
	
	If $\adj(\FF_t)\ge0$ for a very general $t \in T$ then
	there exists $m,n>0$ such that $h^0(\XX_t, {{K_{\FF}}_t}^{\otimes m} \otimes {{N^*_\FF}_t}^{\otimes n}) >0$
	for a very general $t \in T$.  Choose $\varepsilon>0$ small enough and apply
	Proposition \ref{P:Seidenberg} to obtain a Zariski open $U\subset T$ such that  $\FF_t$ has at worst
	$\varepsilon$-canonical singularities for every $t \in U$.
	Remark \ref{R:sing} guarantees that  ${{K_{\FF}}_t}^{\otimes m} \otimes {{N^*_\FF}_t^{\otimes n}}$ is the restriction
	of a line bundle on $\XX$ to $\pi^{-1}(t)$.
	By semi-continuity it follows that $\eff(\FF_t) \ge \frac{n}{m}$ for every $t \in U$.
	
	If instead $\adj(\FF_t)=-\infty$ for a very general $t \in T$
	then Lemma \ref{L:negative open} implies that the same
	holds true for every $t$ in a Zariski open subset of $T$.
	
	In any case, we have just proved that the result is true for the restriction of
	$\FF$ to a non-empty Zariski open subset $U$ of $T$.
	
	The proof continues by passing to  irreducible
	components  of $T\setminus U$. To be more precise let $S$ be a resolution of singularities of one such irreducible component, and let $\YY \to S$ be the base to change of $\XX\to T$ to $S\to T\setminus U \subset T$.
	If \hbox{$v \in H^0(\XX, T_{\XX/T}\otimes K_{\FF})$} is a twisted vector field defining $\FF$
	then it induces an element \hbox{$w \in H^0(\YY, T_{\YY/S}\otimes (K_{\FF})_{\YY})$}. If $w$ is identically zero then
	the points of $S$ do not correspond to foliations and we do nothing. Otherwise, we divide $w$ by the divisorial components of its singular set in order to obtain the defining twisted vector field of a family of foliations
	over $S$.  Since we are back to our original problem, but with a  lower dimensional base, we are done by Noetherian induction.
\end{proof}

\section{Foliations with rational first integrals}\label{S:rationalfirstintegral}

\subsection{Transversely affine and transversely projective foliations} 
A foliation on a projective surface $X$ is called transversely affine if for any rational $1$-form $\omega_0$ defining
$\F$, there exists a  rational $1$-form $\omega_1$ such that
\[
d\omega_0  = \omega_0 \wedge \omega_1  \quad \text{ and } \quad     d\omega_1  = 0 \, .
\]

Similarly, a foliation  $\F$ on $X$ is called transversely projective if for any rational $1$-form $\omega_0$ defining $\F$
there exists rational $1$-forms $\omega_1$ and $\omega_2$ such that
\begin{align*}
d\omega_0 & = \omega_0 \wedge \omega_1  \\
d\omega_1 & = 2\omega_0 \wedge \omega_2 \\
d\omega_2 & =\omega_1 \wedge \omega_2 \, .
\end{align*}

For a thorough discussion about transversely affine and transversely
projective foliations of codimension one on projective manifolds the reader should consult \cite{MR3294560}
and \cite{MR3522824} respectively.

\subsection{Statement of the main result}

This section is devoted to the proof of the following result.

\begin{thm}\label{T:main}
	Let $(\pi: \XX \to T, \FF)$ be an algebraic family of foliations and $g\ge0$ be an integer.
	Let $\Sigma_g \subset T$ be the Zariski closure of the set of parameters corresponding
	to foliations birationally equivalent to a fibration of geometric genus at most $g$.
	Then for every $t\in \Sigma_g$ the foliation $\FF_t$ is transversely projective.
\end{thm}

If one considers the universal family of degree $d$ foliations on $\mathbb P^2$
then one promptly realizes that Theorem \ref{THM:B} is nothing but a particular
case of this more general statement.

\subsection{Example} Before dealing with the proof of Theorem \ref{T:main} let us
analyze the Zariski closure of the set of foliations admitting a rational first integral
in a family derived from Gauss hypergeometric equation.

Whenever $c \notin \mathbb Z$,  Gauss hypergeometric equation
\[
z(1-z)w'' +(c-(a+b+1)z)w'-abw=0,
\]
admits as general solution in a neighborhood of the origin the function
\[
\varphi(z) = C_1 F (a, b, c; z) + C_2 z^{ 1-c} F (a - c + 1, b - c + 1, 2 - c; z),
\]
where $C_1 ,C_2$ are arbitrary constants to be determined by boundary conditions
and
\[
F(a,b,c;z) = 1 + \sum \frac{(a)_n(b)_n}{(c)_n} z^n, \; \; (p)_n:=p(p+1)(p+2)\cdots(p+n-1).
\]
The  change of variable $y(z) = -d \log w(z)$  associates a
Riccati equation/foliation to any second order differential equation.
In this new coordinate the family of foliations induced by Gauss
hypergeometric equation can be written as
\[
\omega =z(1-z)dy- z(1-z)y^2 +(c-(a+b+1)z)y+ab dz \, .
\]
If $\varphi(z)$ is an arbitrary solution of Gauss hypergeometric equation then
$y = - d\log \varphi(z)$ is a solution of the corresponding Riccati equation.
If we choose $c \in \mathbb Q - \mathbb Z$,   $a \in \mathbb Z_{<0}$, and
$b = c - 1 + \beta$ where $\beta \in \mathbb Z_{<0}$ then it is clear from the
explicit form of the solutions that all the leaves of the foliation corresponding to
this choice of parameters are algebraic.
It follows that the set of foliations in this family admitting a rational integral is Zariski dense.

On the one hand, for a very generic choice of parameters, the monodromy group of Gauss hypergeometric
equation is Zariski dense in $\Aut(\mathbb P^1)$, see for instance \cite[Chapter 4, \S 5]{MR1453580}. 
On the other hand,  in the case of transversely affine Riccati foliations the monodromy group 
is not Zariski dense (it must be a solvable subgroup of $\Aut(\mathbb P^1)$. Therefore for a very
generic choice of parameters the foliation defined by $\omega$ is transversely projective, but not
transversely affine. 

To conclude, we point out  that for the
choice of parameters made above the foliations are birationally equivalent to
fibrations by rational curves. Hence one cannot hope to replace
transversely projective by transversely affine in the statement of Theorem \ref{T:main}.

\subsection{Non-isotrivial fibrations}
We   now start the proof of Theorem \ref{T:main}.
We first treat the case of foliations birationally equivalent to non-isotrivial fibrations.

\begin{prop}\label{P:nonisotrivial}
	Let $g\ge 1$ be a natural number and let $(\pi: \XX \to T, \FF)$ be an
	algebraic family of foliations. The Zariski closure in $T$ of the set of
	parameters corresponding to foliations birationally equivalent to
	non-isotrivial fibrations of genus at most $g$  consists of foliations
	admitting rational first integrals.
\end{prop}

\begin{proof}
	According to  \cite[Proposition 2.1]{MR2248154}  it suffices to prove
	that the fibers of the non-isotrivial fibrations in the family belong to a
	bounded family of curves.
	
	For $g=1$ the boundedness is clear since the fibers of non-isotrivial
	elliptic fibration $\FF_t$ are contained in zero sets of sections of
	${K_{\FF}}_t^{\otimes 12}$, see for instance
	\cite[Proposition 4.2]{2016arXiv160405276C}.
	The boundedness of fibers of non-isotrivial fibrations of genus
	$g\ge 2$ is guaranteed by Theorem \ref{THM:A}.
\end{proof}

\subsection{Isotrivial fibrations of adjoint general type}
For isotrivial fibrations of adjoint general type  the situation is  better when
compared to non-isotrivial fibrations as there is no need to bound the
genus in order to obtain boundedness of the leaves.

\begin{prop}\label{P:isotrivial}
	Let $(\pi: \XX \to T, \FF)$ be an algebraic family of foliations.
	The Zariski closure in $T$ of the set of parameters corresponding to
	foliations of adjoint general type  birationally equivalent to
	isotrivial fibrations  consists of foliations admitting rational first integrals.
\end{prop}

\begin{proof}
	If $\F$ is an isotrivial fibration of adjoint general type on a projective surface $X$ then $\F$ has
	Kodaira dimension one and the Iitaka fibration of $\KF$ is an isotrivial
	fibration of genus $g\ge 2$.
	According to \cite[Proposition 4.10]{2016arXiv160405276C} there are
	at least two linearly independent sections $\sigma_1, \sigma_2$ of $\KF^{\otimes k}$ for some $k\le 42$.
	Consider the rational map $f = (\sigma_1: \sigma_2) : X \dashrightarrow \mathbb P^1$
	defined by them. The foliation $\G$ defined by $f$ coincides with the foliation
	defined  by the Iitaka fibration of $\KF$. Its normal bundle is of the form $N_{\G} =
	f^* T_{\mathbb P^1} \otimes \mathcal O_X( - \Delta) = \KF^{\otimes 2k} \otimes \mathcal O_X( - \Delta)$
	where $\Delta$ is an effective divisor.
	Since the leaves of $\F$ are contained in fibers of the Iitaka fibration
	of $\KG$,  we  repeat the  argument to  obtain the existence of a $k'\le 42$
	such that the leaves of $\F$ are contained in zero set of sections of
	$\KX^{\otimes k'} \otimes \KF^{\otimes 2 k' k } \otimes \mathcal O_X(- k' \Delta)$.
	This suffices to prove the boundedness of the leaves of foliations in a  family
	having adjoint general type and birationally equivalent to isotrivial fibrations.
\end{proof}

\subsection{First integrals and transverse structures}

\begin{prop}\label{P:special is proj}
	Let $\F$ be a foliation on a projective surface $X$.
	If $\adj(\F)<2$ then $\F$ is a transversely projective foliation.
	Moreover, if $\adj(\F) \in \{ 0, 1 \}$  then
	$\F$ is a transversely affine foliation.
\end{prop}

\begin{proof}
	This is a straightforward consequence of the classification.
	If $\F$ has adjoint dimension zero then it is  birationally equivalent
	to a finite quotient of a foliation defined
	by a closed rational $1$-form. Since the property of being
	transversely affine is invariant by dominant rational maps, $\F$ is transversely affine.
	If $\F$ has adjoint dimension one then $\F$ is either a fibration
	(and therefore clearly transversely affine) or $\F$ is a turbulent foliation
	which is well-known to be transversely affine (see for instance \cite[Proposition 22]{MR2029287}).
	Finally if $\F$ has negative adjoint dimension then it is either a fibration,
	a Riccati foliation, or a finite quotient of a Riccati foliation. In any case we
	have that $\F$ is a transversely projective foliation.
\end{proof}

\begin{prop}\label{P:tpfam}
	Let $(\pi: \XX \to T, \FF)$ be an algebraic family of foliations. If for a very general closed point
	$t_0 \in T$ the foliation $\FF_{t_0}$ is a transversely projective foliation then  for every closed 	
	point $t \in T$ the foliation $\FF_t$ is a transversely projective foliation.
	Similarly, if for a very general closed point
	$t_0 \in T$ the foliation $\FF_{t_0}$ is a transversely affine foliation then
	for every closed point $t \in T$ the foliation $\FF_t$ is a transversely affine foliation.
\end{prop}

\begin{proof}
	We can interpret the family of foliation as a single foliation defined over the function field
	${\mathbb C(T)}$. By assumption, this foliation is  transversely projective.
	Hence there exists a triplet $(\omega_0, \omega_1, \omega_2)$
	of rational differential $1$-forms with coefficients in $\overline{\mathbb C(T)}$,
	the algebraic closure of $\mathbb C(T)$, satisfying the equations
	\begin{align*}
	d\omega_0 & = \omega_0 \wedge \omega_1  \\
	d\omega_1 & = 2\omega_0 \wedge \omega_2 \\
	d\omega_2 & =\omega_1 \wedge \omega_2 \, .
	\end{align*}
	and such that $\omega_0$ is a $1$-form differential form defined over $\mathbb C(T)$ which defines $\FF$. According
	to \cite[Lemma 3.2]{MR1955577} we can assume that $\omega_1, \omega_2$ are also defined over $\mathbb C(T)$ ( no need to
	pass to the algebraic closure). Therefore, over $\mathbb C$, we have the equations
	\begin{align*}
	d\omega_0\wedge d\pi & = \omega_0 \wedge \omega_1\wedge d\pi  \\
	d\omega_1\wedge d\pi & = 2\omega_0 \wedge \omega_2\wedge d\pi \\
	d\omega_2\wedge d\pi & =\omega_1 \wedge \omega_2\wedge d\pi \, .
	\end{align*}
	If $t \in T$ is such that $\pi^{-1}(t)$ is not contained in the polar set of $(\omega_i)_{\infty}$ for $i = 0 , 1 ,2$
	nor in the zero set of $\omega_0$ then the restriction of the triple $(\omega_0, \omega_1, \omega_2)$ to the fiber over $t$ defines
	a (singular) projective structure for the foliation $\FF_t$ on $X_t = \pi^{-1}(t)$.
	
	Let us  fix $t_0 \in T$ such that $X_0 = \pi^{-1}(t_0)$ is contained in the polar set of $\omega_i$ ($i=0,1,2$) or in the zero set of $\omega_0$  and let $f \in \pi^* \mathcal O_{T,t_0}$ be a
	rational function on $X_0$ corresponding to a
	generator of the maximal ideal of $\mathcal O_{T,t_0}$. Notice that we can replace the triplet $(\omega_0,\omega_1, \omega_2)$ by $(f^{k} \omega_0, \omega_1, f^{-k}\omega_2)$.
	Thus, there is no loss of generality in assuming that $\pi^{-1}(t_0)$ is not contained in $(\omega_0)_{\infty} \cup (\omega_0)_0$.

	For $i = 0 , 1 , 2$, let $a_i$ be the order of $\omega_i$ along $X_0$ and set $\alpha_i =  \Res_{X_0} f^{-a_i} \omega_i \wedge \frac{df}{f}$.
	As mentioned above we will assume that $a_0=0$ and, therefore,  $\alpha_0$ is just the restriction of $\omega_0$ to the
	fiber $X_0$.
	
	If $a_1$ is negative then, comparing the  orders along $X_0$ of $d\omega_0\wedge df $ and of $\omega_0 \wedge \omega_1\wedge df$, we deduce that
	$\alpha_0 \wedge \alpha_1 = 0$ and we can write $\alpha_0 = g \alpha_1$ for some rational function $g \in \mathbb C(X_0)$. Let $G \in \mathbb C(\XX)$
	be a rational function on $\XX$ extending $g$. According to formula (14) of \cite{MR2324555} we can replace the
	triplet $(\omega_0, \omega_1, \omega_2)$ by the triplet
	\[
	\left(\omega_0 , \,  \omega_1 - f^{-a_1} G \omega_0, \, \omega_2 + f^{-a_1}G \omega_1 + f^{-2a_1}G^2 \omega_0 - f^{-a_1}dG\right).
	\]
	This increases $a_1$. After a finite number of changes we may assume that $a_0=0$ and $a_1 \ge 0$.
	
	Finally, if  $a_2$ is negative and $a_1>0$ then $\alpha_0$ is closed and
	it is clear that $\FF_{t_0}$ is transversely projective. If instead $a_2<0$ and $a_0=a_1=0$  then comparing the  orders along $X_0$ of $d\omega_1\wedge df $ and  $\omega_0 \wedge \omega_2\wedge df$ we deduce that
	$\alpha_0 \wedge \alpha_2 = 0$. Thus we can write $\alpha_2= h\alpha_0$ for a suitable rational function $h \in \mathbb C(X_0)$. From the equation
	$d\omega_2\wedge df =\omega_1 \wedge \omega_2\wedge df$ we deduce that $d\alpha_2 = \alpha_1 \wedge \alpha_2$. Combining these two identities we obtain
	\[
	d(h\alpha_0) = \alpha_1 \wedge (h \alpha_0)  \implies d\alpha_0 = (\alpha_1 - \frac{dh}{h}) \wedge \alpha_0.
	\]
	Finally, comparing this identity with  $d\alpha_0 = \alpha_0 \wedge \alpha_1$ (first equation) we obtain that $d \alpha_0 = -(1/2) \frac{dh}{h} \wedge \alpha_0$.
	Thus $\FF_{t_0}$ is transversely projective also in this case.
\end{proof}

\subsection{Proof of Theorem \ref{T:main} (and of  Theorem \ref{THM:B})}
Let $(\pi: \XX \to T, \FF)$ be an algebraic family of foliations and $g\ge0$ be an integer.
We want to prove that the Zariski closure of $\Sigma_g \subset T$ (subset  parametrizing
foliations with rational first integral of genus at most $g$)
corresponds to transversely projective foliations.

If a very general member of the family, say $\F_t$, is not of  adjoint general type
then Proposition \ref{P:special is proj} implies that $\F_t$ is transversely projective. We can apply
Proposition \ref{P:tpfam} to conclude that every foliation in the family is also transversely projective.

If instead a very general member is of adjoint general type then we will argue  as in the proof of Theorem
\ref{T:non accumulation} to obtain a non-empty Zariski open subset of $T$ such that every foliation parametrized by this
subset is of adjoint general type.

Proposition \ref{P:Seidenberg}  allows us to assume
the existence of a non-empty Zariski  open subset $U_0 \subset T$ that for a very general $t \in U_0$, the foliation $\FF_t$ has canonical singularities.  Since $\mathbb C$ is uncountable
we also know that there exists $n, m>0$ and an open subset $U_1 \subset T$ such that for every $t \in U_1$, the linear system $|{{K_{\FF}}_t}^{\otimes m} \otimes {{N^*_\FF}_t}^{\otimes n}|$
defines a rational map with two dimensional image. Notice that there may exist foliations in $U_0 \cap U_1$ which are not of adjoint general type
because of the presence of non-canonical singularities.  To remedy this we take $\varepsilon >0$ sufficiently small  in order to obtain from Lemma \ref{L:epsilon open}
a non-empty Zariski open $U_2\subset T$  such that $\FF_t$ has $\varepsilon$-canonical singularities. Every foliation parametrized by non-empty Zariski open  $U=U_0 \cap U_1 \cap U_2$ is
of adjoint general type.

Propositions \ref{P:nonisotrivial} and \ref{P:isotrivial} imply that the Zariski closure in $T$ of
$\Sigma_g \cap U$ corresponds to foliations with rational first integrals. The Theorem follows  by Noetherian induction.  \qed

\bibliographystyle{amsalpha}
\bibliography{references}{}

\end{document}